\documentclass[a4paper,reqno,11pt]{article}
\usepackage{t1enc}
\usepackage{ifthen}
\usepackage{graphicx}
\usepackage{amsmath,amstext,amssymb,amsopn,amsthm,color}

\newcommand{\kernel}{k}

\newcommand{\kt}{k(t,x,y)}

\renewcommand{\(}{\left(}
\renewcommand{\)}{\right)}

\renewcommand{\[}{\left[}
\renewcommand{\]}{\right]}

\newcommand{\R}{\mathbf{R}}

\newcommand{\pr}{\mathbf{P}}
\newcommand{\ex}{\mathbf{E}}
\newcommand{\E}{\mathbf{E}}

\newcommand{\bs}{\backslash}

\theoremstyle{plain}
\newtheorem{theorem}{Theorem}
\newtheorem{lemma}{Lemma}
\newtheorem{corollary}{Corollary}
\newtheorem{proposition}{Proposition}

\theoremstyle{definition}

\theoremstyle{remark}

\title {Feeling boundary by Brownian motion in a ball
\footnotetext{2000 MS Classification:
    35K08, 60J65. {\it Key words and phrases}: heat kernel,  ball, asymptotics, Laplacian, Brownian motion} 
    \footnotetext{ The author was supported partly by the National
Science Centre grant no. 2015/18/E/ST1/00239 and partly by the Faculty of Pure and Applied Mathematics of Wroc{\l}aw University of Science
and Technology (0402/0051/18).
       }
\author{ G. Serafin\\ Faculty of Pure and Applied Mathematics\\ Wroc{\l}aw University of Science and Technology
}}
\date{}

\begin{document}
\maketitle

\begin{abstract}
We establish short-time asymptotics with rates of convergence for the Laplace Dirichlet  heat kernel in a ball. So far, such results were only known in simple cases where explicit formulae are available, i.e., for sets as half-line, interval and their products.  Presented asymptotics may be considered as a complement  or a generalization of the famous "principle of not feeling the boundary" in case of a ball.  Following the metaphor, the principle reveals when the process does not feel the boundary, while we describe what happens when it starts feeling the boundary.
\end{abstract}

\section{Introduction}
\label{sec:introduction}
Let  $k(t,x,y)=(4\pi t)^{-n/2}e^{-|x-y|^2/4t}$, $n\geq2$, be the  global heat kernel for the Laplacian in $\R^n$. For an open domain $D\subset \R^n$ we denote by $k_D(t,x,y)$ the Dirichlet heat kernel for the  Laplacian in $D$. Z. Ciesielski has proven in \cite{Ciesielski:1966} that if the interval 
$I(x,y)=\{z=\alpha x+(1-\alpha)y: \alpha\in[0,1]\}$
connecting $x$ and $y$,  is contained in $D$, then the heat kernel $k_D(t,x,y)$ satisfies the well known ''principle of not feeling the boundary'' (introduced by M. Kac in \cite{Kac:1951})
\begin{align}\label{eq:nfb}\lim_{t\rightarrow0}\frac{k_D(t,x,y)}{k(t,x,y)}=1.\end{align}
In \cite{vdB}, M. van den Berg improved it by providing the following rate of convergence
\begin{align}\label{eq:vdB}
k(t,x,y)\geq k_D(t,x,y)\geq k(t,x,y)\(1-e^{-\rho^2/t}\sum_{k=1}^n\frac{2^k}{(k-1)!}\(\frac{\rho^2}{t}\)^{k-1}\),
\end{align}
where $\rho$ is the distance between $I(x,y)$ and the boundary $\partial D$ of the domain $D$, i.e.
$$\rho=\inf_{\substack{w\in I(x,y)\\ z\in \partial D}}|w-z|.$$
 This kind of short-time asymptotic behaviour of the Dirichlet heat kernels has been studied and generalized in many papers, see e.g. \cite{ H2, Berg:1990, Varadhan:1967}. However, there is still no answer to a very natural question:  what is the limit in \eqref{eq:nfb} if $x$ or $y$ is getting close to the boundary of $D$, or following the metaphorical convention: what happens when the process starts feeling the boundary? The answer in known only in few elementary cases, e.g. for  a half-space or an interval, where simple explicit formulas of heat kernels are available.

Research on short-time boundary behaviour of Dirichlet heat kernels has a long history, but  concerns mainly estimates, and not asymptotics (see, among others,  \cite{CheegerYau:1981, Davies:1987, GrigorSaloff:2002, MR2807275, MS,  Zhang:2002, Zhang:2003, U}). In particular, let us recall very general bounds of E. B. Davies \cite{Davies:1987} (the upper bound) and  \mbox{Q. S. Zhang} \cite{Zhang:2002} (the lower bound),  which state that for any bounded domain $D\subset \R^n$ there are constants $c_1,c_2,c_3,c_4>0$ such that 

\begin{eqnarray*}
  \label{eq:Davies:est}
  c_1\left(\frac{\delta_D(x)\delta_D(y)}{t}\wedge 1\right)k(c_2t,x,y)  \leq k_D(t,x,y) \leq c_3\left(\frac{\delta_D(x)\delta_D(y)}{t}\wedge 1\right)k(c_4t,x,y).
\end{eqnarray*}
for every $x,y\in D$ and $t<T$ for some $T>0$. Here, $\delta_D(x)$ denotes the distance of $x$ to the boundary of $D$. The result is very powerful, but also very imprecise from the point of view of asymptotics. Its main disadvantage is that the time variable is multiplied by different constants in lower and upper bounds, which means that the exponential behaviours differ significantly for large values of $|x-y|^2/t$, and consequently both bounds become  completely incomparable. This inaccuracy has been recently removed in \cite{MS} in case of a ball. Precisely, let $B=B(0,1)$ be a unit ball centered at the origin and denote by $k_B(t,x,y)$ the heat kernel of $B$. Then,    for every  $T>0$ there exists a constant $C=C(n,T)>1$  such that 
\begin{eqnarray}
\label{eq:MS}
\frac{1}{C}\,  h(t,x,y)k(t,x,y)\leq k_B(t,x,y)\leq C\, h(t,x,y)k(t,x,y)
\end{eqnarray}
for every $x,y\in B$ and $t<T$, where 
\begin{eqnarray}
\label{eq:htxy:est}
  h(t,x,y) = \left(1\wedge\frac{\delta_B(x)\delta_B(y)}t\right)+\left(1\wedge\frac{\delta_B(x)|x-y|^2}t\right)\left(1\wedge\frac{\delta_B(y)|x-y|^2}t\right)\/.
\end{eqnarray}
 This estimate is a step forward and gives us new information about the behaviour of the heat kernel near the boundary. Nevertheless, it still does not enable us to obtain precise asymptotics of the quotient $k_B(t,x,y)/k(t,x,y)$. See also \cite{B, BM, MSZ, NSS1, NSS2, S, U} for other recent research on accurate exponential behaviour of heat kernels and densities of joint distribution of first hitting time and place.

The goal of the paper is to derive uniform short-time asymptotics  of the heat kernel $k_B(t,x,y)$ of the ball $B(0,1)$ as well as  to provide rates of convergence. Note that long-time asymptotics follow from the general theory (see \cite{Davies:1987}, \cite{DaviesSimon:1984}) or the series representation given in \cite H. For small times, even though the estimates \eqref{eq:MS} are known, it is not clear what is the main factor impacting on asymptotics. One can see that when $|x-y|^2/\sqrt t$ is small then the left-hand side component in \eqref{eq:htxy:est} dominates the other one, but when $|x-y|^2/\sqrt t$ is large then the right-hand side component is not the dominating one  in some range of arguments (e.g. for $\delta_B(y)<|x-y|^2<\delta_B(x)$). It turns out that the proper quantity that drives the short-time behaviour of the heat kernel of a ball is $\delta_B\(\frac{x+y}2\)/\sqrt t$, which is comparable to $(|x-y|^2+\delta_B(x)+\delta_B(y))/\sqrt t$, see \eqref{eq:parallel}. The main results of the paper are  Theorems \ref{thm:Thm1} and \ref{thm:Thm2}, where asymptotics of the heat kernel as $\delta_B\(\frac{x+y}2\)/\sqrt t$ tends to infinity or zero, respectively, are presented.
\begin{theorem}\label{thm:Thm1}
There are constants $C,M>0$ depending only on $n$ such that for $\frac{\delta_B\(\frac{x+y}2\)}{\sqrt t}>M$ we have 
\begin{align*}
&\Bigg|k_B(t,x,y) -\left(1-e^{-2\delta_{H_x}(x)\,\delta_{H_x}\left(\frac{x+y}2\right)/t}\right)\left(1-e^{-2\delta_{H_y}(y)\,\delta_{H_y}\left(\frac{x+y}2\right)/t}\right)k(t,x,y)\Bigg|\\
&\leq C \sqrt{\frac{\sqrt t}{\delta_B\(\frac{x+y}2\)}}k_B(t,x,y),
\end{align*}
where $\delta_{H_z}(w)$, $w,z\in B(0,1)$,  denotes the distance between $w\in B(0,1)$ and the hyperplane tangent to the ball at the point $z/|z|$. In particular, it holds $\delta_{H_z}(z)=\delta_B(z)$.
\end{theorem}
Let us note that the factors appearing in the asymptotical form come from the form of the heat kernel of a half-space. Precisely, we can write
$$\left(1-e^{-2\delta_{H_x}(x)\,\delta_{H_x}\left(\frac{x+y}2\right)/t}\right)=\frac{k_{H_x}\(t/2,x,\frac{x+y}2\)}{k\(t/2,x,\frac{x+y}2\)}.$$
One may therefore interpret this in the following way: if the distance from  the middle point between $x$ and $y$ to the boundary is  much bigger than $\sqrt t$, then the process traveling from $x$ to $y$ is, before reaching neighbourhood of  the middle point, similar to the process in the half-space $H_x$, and to the process in $H_y$ after reaching the neighbourhood of the midpoint. 
\begin{theorem}\label{thm:Thm2}
There are constants $C,m_1,m_2>0$ depending only on $n$ such that for $t<m_1$ and $\frac{\delta_B\(\frac{x+y}2\)}{\sqrt t}<m_2$ we have
$$\left|k_B(t,x,y)-\frac{\delta_B(x)\delta_B(y)}tk(t,x,y)\right|\leq C  \(\sqrt t+\sqrt{\frac{\delta_B\(\frac{x+y}{2}\)}{\sqrt t}}\)k_{B}(t,x,y),\ \ \ \ x,y\in B.$$
\end{theorem}
In case of Theorem \ref{thm:Thm2} we could say that if the distance from  the middle point between $x$ and $y$ to the boundary is  much shorter than $\sqrt t$, then the process is similar to a process living in a suitably chosen half-space (see Lemma \ref{lem:2.2}). In fact,  $\delta_B(x)\delta_B(y)/t$ may be replaced by $1-e^{-\delta_B(x)\delta_B(y)/t}$, which  shows more accurately the connection with the form of the half-space heat kernel. 

Since short-time asymptotics of Dirichlet heat kernels describing boundary behaviour have been known  only for simple sets as half-space or an interval, there were no methods developed for solving such problems. Some ideas are taken from the recent paper \cite{MS}, where estimates have been obtained, however, providing asymptotics requires much more effort and  care of details. As mentioned before, even knowing the estimates, it was not clear what kind of asymptotics one should expect. The first and crucial step in proving Theorems \ref{thm:Thm1} and \ref{thm:Thm2} was to approximate the heat kernel of a ball by the heat kernel $k_{H_x}(t,x,y)$ of the half-space $H_x$ in suitable range of argument. It has been achieved by combination of strong Markov property and $n$-dimensional analysis. In particular, the density of a convolution of two inverse-gamma distributions has been estimated, since it appears naturally when employing strong Markov property for Brownian motion.    
A lot of geometrical arguments has been used as well.
Then, in view of explicit and compact form of $k_{H_x}(t,x,y)$, 
furher approximations were possible. Both, the result and the methods presented in the paper may be applied in more general setting as e.g. in estimating  Dirichlet heat kernels of $\mathcal C^{1,1} $ domains.
 
Let now $q^B_x(t,z)$ be the density of the joint distribution of first hitting time and place of Brownian motion exiting a ball.  The representation of 
$q^B_x(t,z)$ as a derivative of $k_B(t,x,y)$  in the inward norm direction (see \eqref{eq:hx:diff}) implies $q^B_x(t,z)=\lim_{h\nearrow1}\(k(t,x,hz)/\delta(hz)\)$.  Thus, dividing inequalities in Theorems \ref{thm:Thm1} and \ref{thm:Thm2} by $\delta_B(y)$ and letting it tend to zero,  we directly obtain the below-given  asymptotics of $q^B_x(t,z)$. So far, only estimates of $q^B_x(t,z)$ \cite{MS} and asymptotics of density of hitting time \cite{S} have been known. 
\begin{corollary} There are constants $C_1,C_2,m_1,m_2, M>0$ depending only on $n$ such that for $\frac{\delta_B\(\frac{x+z}2\)}{\sqrt t}>M$ we have 
\begin{align*}
&\Bigg|q^B_x(t,z) -\left(1-e^{-2\delta(x)\,\delta_{H_x}\left(\frac{x+z}2\right)/t}\right)\frac{2\delta_{H_z}\left(\frac{x+z}2\right)}tk(t,x,y)\Bigg|\leq C_1 \sqrt{\frac{\sqrt t}{\delta_B\(\frac{x+z}2\)}}q^B_x(t,z),
\end{align*}
while for  $t<m_1$ and $\frac{\delta_B\(\frac{x+z}2\)}{\sqrt t}<m_2$ it holds
$$\left|q^B_x(t,z)-\frac{\delta(x)}tk(t,x,z)\right|\leq C_2  \(\sqrt t+\sqrt{\frac{\delta_B\(\frac{x+z}{2}\)}{\sqrt t}}\)q^B_x(t,z).$$
\end{corollary}

The paper is organized as follows. In Section 2 we introduce notation and provide some useful facts concerning Brownian motion. Sections 3 and 4 are devoted to the proofs of Theorem \ref{thm:Thm1} and Theorem \ref{thm:Thm2}, respectively.  Finally, in Appendix we  gather several technical lemmas that are exploited in the proofs.
\section{Preliminaries}
\subsection{Notation}\label{sec:notation}
We write $f\lesssim g$ whenever there exists a constant $c>0$ depending only on a dimension $n$ such that $f<cg$ holds for the indicated range of the arguments of functions $f$ and $g$. If $f\lesssim g$ and $g\lesssim f$, the we write $f\approx g$. 

By $|x|$ we denote the Euclidean norm of a point $x\in \R^n$. We write $B_k(x_0,r)=\{x\in\R^k:|x-x_0|<r\}$ for a $k$-dimensional ball of a radius $r>0$ centered at $x_0\in\R^k$. In the basic case $x_0=0$, $r=1$ and $k=n$ we simply denote $B=B_n(0,1)$. 

For $x\in B$, $x\neq 0$, we write $P_x$ for hyperplane tangent to $B$ at the point $\frac{x}{|x|}$. The half-space bounded by $P_x$ and containing the ball $B$ will be denoted by $H_x$.  Additionally, by $P_{xy}$ we denote the hyperplane   that contains  $x/|x|, y/|y|\in \partial B$ and such that it is perpendicular to the vector $\frac12\(\frac x{|x|}+\frac y{|y|}\)$. For example, if $x=(x_1,x_2,0,...,0)$, $y=(y_1,y_2,0,...,0)$ and $l$ is the line in $\R^2$ containing $(x_1,x_2)/|x|$ and $(y_1,y_2)/|y|$ then $P_{xy}=l\times \R^{n-2}$. Furthermore, we  define $H_{xy}$ as the half-space bounded by $P_{xy}$ and containing $x$ and $y$.

For a domain $D\subset \R^n$ and $x\in D$ we write $\delta_D(x)$ for a distance of $x$ to the boundary $\partial D$. As previously, we shorten the notation in the case  $D=B$ and just write $\delta(x) = \delta_{B}(x)=1-|x|$. Let us note that the distance of a middle point between $x$ and $y$ to the boundary of $B$ may be estimated as follows (\cite{MS}, formula (2.2)) 
\begin{eqnarray}
  \label{eq:parallel}
\delta\left(\frac{x+y}{2}\right)\geq \frac{|x-y|^2}{8}+\frac{\delta(x)}{4}+\frac{\delta(y)}{4} \geq \frac12 \delta\left(\frac{x+y}{2}\right).
\end{eqnarray}
In particular, this implies
\begin{align}\nonumber
\left|\frac{x}{|x|}-\frac{y}{|y|}\right|&\leq \delta(x)+|x-y|+\delta(y)\leq \sqrt{\delta(x)}+\sqrt{|x-y|^2}+\sqrt{\delta(y)}\\\label{eq:x_0-y_0}
&\leq \sqrt{3\(\delta(x)+|x-y|^2+\delta(y)\)}\leq 2\sqrt6\,\sqrt{\delta\(\frac{x+y}2\)},
\end{align}
where we used the inequality between arithmetic mean and root mean square.
\subsection{Brownian motion}
We consider $n$-dimensional, $n\geq 2$, Brownian motion $W=(W(t))_{t\geq 0}=(W_1(t),...,W_n(t))_{t\geq 0}$ starting from $x\in\R^n$ and we denote by $\pr^x$ and $\ex^x$ the corresponding probability law and the expected value, respectively. Obviously $\pr^x$ is absolutely continuous with respect to the Lebesgue measure and $\kt$ is the corresponding transition probability density.

For a general  domain $D\subset \R^d$ we define the first exit time from $D$ by
\begin{eqnarray*}
   \tau_D = \inf\{t>0: W(t)\notin D\}\/.
\end{eqnarray*}
We write $\kernel_D(t,x,y)$ for the transition probability density for  Brownian motion $W^{D}=(W^D(t))_{t\geq 0}$ killed upon leaving a set $D$.  The relation between $\kernel_D(t,x,y)$ and $\kt$ together with the joint distribution of $(\tau_D, W(\tau_D))$ is described by the Hunt formula
\begin{eqnarray}
   \label{eq:Hunt:general}
	\kernel_D(t,x,y) = \kt-\ex^x[t>\tau_D, k(t-\tau_D,W(\tau_D),y)]\/,\quad x,y\in D\/,\quad t>0\/.
\end{eqnarray}
Denoting the density function of $(\tau_D,W(\tau_D))$ by $q^D_x(t,z)$, the Hunt formula takes the form
\begin{eqnarray}
\label{eq:Hunt:ball}
   k_D(t,x,y) = \kt-\int_0^t \int_{\partial D}\kernel(t-s,z,y)q^D_z(s,z)dsd\sigma(z)\/,
\end{eqnarray}
where $\sigma$ is the surface measure on $\partial D$. Note also that if $D$ is a $\mathcal C^3$ then we can recover $q^D_x(t,z)$ from $k_D(t,x,y)$ by differentiating it in the inward norm direction (see \cite{H})
\begin{eqnarray}
 \label{eq:hx:diff}
  q^D_x(t,z) = \dfrac{\partial}{\partial n_z}\kernel_D(t,x,z)\/,\quad x\in D\/,z\in\partial D\/,t>0\/.
\end{eqnarray}
In particular, in case of the  ball $B$ we may use the estimates \eqref{eq:MS} and get
\begin{align}\label{eq:qest}
q^B_x(t,z)=\lim_{h\downarrow0}\frac{k_B(t,x,(1-h)z)}{h}\approx \(\frac{\delta(x)}t+\frac{|x-z|^2}t\left(1\wedge\frac{\delta(x)|x-z|^2}t\right)\)k(t,x,z).
\end{align}
Due to the reflection principle, in case of a half-space a simple explicit formula may be derived. More precisely, for $H=\{x\in \R^n: x_1>0\}$  we have
\begin{eqnarray*}
   k_H(t,x,y) = \kt-k(t,x,(-y_1,y_2,...,y_n))=\(1-e^{-\frac{x_1y_1}{t}}\)\kt.
\end{eqnarray*} 
Generally, by rotational and translational invariance of Brownian motion,  for any half-space $H$ the following holds
\begin{align}\label{eq:kHform}
   k_H(t,x,y) &=\(1-e^{-\frac{\delta_H(x)\delta_H(y)}{t}}\)\kt\\
	\label{eq:kHest}&\approx \(1\wedge \frac{\delta_H(x)\delta_H(y)}{t}\)\kt,\ \ \ x,y\in H,\ \ t>0.
\end{align}

\section{Proof of Theorem \ref{thm:Thm1}}

First, we prove the assertion of Theorem \ref{thm:Thm1} in case when $\delta(y)/\sqrt t\rightarrow\infty$ with a bit different form of the  convergence rate.

\begin{lemma}\label{lem:1.1}
There is $M>0$ such that for $\frac{\delta(y)}{\sqrt t}>M$ we have
$$|k_B(t,x,y)-k_{H_x}(t,x,y)|\lesssim \frac t{(\delta(y))^2}k_{B}(t,x,y).$$
\end{lemma}
\begin{proof}

 If $\delta(x)>\frac1{64}\delta(y)$, then for  $\delta(y)/\sqrt t$ large enough inequality \eqref{eq:vdB} implies
$$|k_B(t,x,y)-k_{H_x}(t,x,y)|\lesssim \frac t{(\delta(y))^2},$$
so we assume  $\delta(x)< \frac1{64}\delta(y)$. In particular, this implies  $|x-y|\geq \frac12 \delta(y)$.
It is shown in the proof of Proposition 3 in \cite{MS} that for $y\in B\left(\frac{15}{16}\frac{x}{|x|},\frac{1}{16}\right) $ the following estimate holds
\begin{align}\label{eq:k-k}
|k_B(t,x,t)-k_{H_x}(t,x,y)|\lesssim \left(e^{-\frac{|x-y|^2}{16t}}+\frac{t}{\delta(y)^2}\right)k_{H_x}(t,x,y),
\end{align}
which, under current assumptions, gives us for  $y\in B\left(\frac{15}{16}\frac{x}{|x|},\frac{1}{16}\right)$ 
\begin{align}\label{eq:ineq1/16}
|k_B(t,x,t)-k_{H_x}(t,x,y)|\lesssim \frac{t}{\delta(y)^2}k_{H_x}(t,x,y).
\end{align}
 Next, we will get rid of the assumption $y\in B\left(\frac{15}{16}\frac{x}{|x|},\frac{1}{16}\right)$. Using the inequality $\delta(x)< \frac1{64}\delta(y)$ we get for any $y\in B$
\begin{align*}
\left|\frac1{32}y+\frac{31}{32}x-\frac{15}{16}\frac{x}{|x|}\right|&\leq\left|\frac1{32}y+\frac{31}{32}x-\left(\frac1{32}y+\frac{31}{32}\frac x{|x|}\right)\right|+\left|\frac1{32}y+\frac{31}{32}\frac x{|x|}-\frac{15}{16}\frac{x}{|x|}\right|\\
&=\frac{31}{32}(1-|x|)+\frac1{32}\left|y-\frac{x}{|x|}\right|\leq \frac{31}{32}(1-|x|)+\frac1{32}\sup_{z\in\partial B}|y-z|\\
&= \frac{31}{32}\delta(x)+\frac1{32}(2-\delta(y))\leq \frac1{16}-\frac1{64}\delta(y)+\(\delta(x)-\frac1{64}\delta(y)\)\\
&\leq \frac1{16}-\frac1{64}\delta(y),
\end{align*}
which implies
\begin{align}\label{eq:BsubsetB}
B_n\left(\frac1{32}y+\frac{31}{32}x,\,\frac1{64}\delta(y)\right)\subset B_n\(\frac{15}{16}\frac{x}{|x|},\frac1{16}\).
\end{align} 
We therefore apply Chapman-Kolmogorov identity in the following way
\begin{align*}
k_B(t,x,y)-k_{H_x}(t,x,y)&=\int_{B_n\left(\frac1{32}y+\frac{31}{32}x,\,\frac1{64}\delta(y)\right)}+\int_{B\bs B_n\left(\frac1{32}y+\frac{31}{32}x,\,\frac1{64}\delta(y)\right)}\\
&\ \ \ \ k_B(t/32,x,z)k_B(31t/32,z,y)-k_{H_x}(t/32,x,z)k_{H_x}(31t/32,z,y)dz\\
&\ \ \ - \int_{H_x\bs B}k_{H_x}(t/32,x,z)k_{H_x}(31t/32,z,y)dz\\[10pt]
&=I_1+I_2-I_3.
\end{align*}
Corollary \ref{cor:CKHH} gives us
\begin{align}\nonumber
|I_2|, |I_3|&\leq 2\int_{H_x\bs B_n\left(\frac1{32}y+\frac{31}{32}x,\,\frac1{64}\delta(y)\right)} k_{H_x}(t/32,x,z)k_{H_x}(31t/32,z,y)dz\\\label{eq:aux6}
&\lesssim k_{H_x}(t,x,y)\exp\left(-c\frac{(\delta(y))^2}{t}\right)\lesssim \frac t{(\delta(y))^2}k_{H_x}(t,x,y),
\end{align}
for some constant $c>0$.
In order to deal with the integral $I_1$, let us note that  $\delta\left(\frac1{32}y+\frac{31}{32}x\right)>\delta(y)/32$.  Hence, for $z\in B_n\left(\frac1{32}y+\frac{31}{32}x,\,\frac1{64}\delta(y)\right)$ we have $\delta(z)>\delta(y)/64$ and, by \eqref{eq:vdB},
$$|k_B(31t/32,z,y)-k_{H_x}(31t/32,z,y)|\lesssim\frac t{(\delta(y))^2}k_{H_x}(31t/32,x,z).$$
Furthermore, by \eqref{eq:BsubsetB} and \eqref{eq:ineq1/16}, we obtain
$$|k_B(t/32,x,z)-k_{H_x}(t/32,x,z)|\lesssim\frac t{(\delta(y))^2}k_{H_x}(t/32,x,z).$$
Thus, we get
\begin{align}\nonumber
&|k_B(t/32,x,z)k_B(31t/32,z,y)-k_{H_x}(t/32,x,z)k_{H_x}(31t/32,z,y)|\\\nonumber
&\leq |\(k_B(t/32,x,z)-k_{H_x}(t/32,x,z)\)k_B(31t/32,z,y)|\\\nonumber
&\ \ \ +|k_{H_x}(t/32,x,z)\(k_B(31t/32,z,y)-k_{H_x}(31t/32,z,y)\)|\\\label{eq:aux7}
&\leq \frac t{(\delta(y))^2}k_{H_x}(t/32,x,z)k_{H_x}(31t/32,z,y),
\end{align}
and consequently
\begin{align*}
|I_1|\lesssim \frac t{(\delta(y))^2}\int_{H_x}k_{H_x}(t/32,x,z)k_{H_x}(31t/32,z,y)dz\lesssim  \frac t{(\delta(y))^2}k_{H_x}(t,x,y).
\end{align*}
This together with \eqref{eq:aux6} let us write 
$$|k_B(t,x,y)-k_{H_x}(t,x,y)|\lesssim \frac t{(\delta(y))^2}k_{H_x}(t,x,y).$$
In particular, for $\delta(y)/\sqrt t$ large enough we have $k_B(t,x,y)\approx k_{H_x}(t,x,y)$, so we can replace $k_{H_x}(t,x,y)$ by $k_{B}(t,x,y)$ in the estimate above. 
\end{proof}

\begin{proof}[Proof of Theorem \ref{thm:Thm1}]
Let us  decompose $k_B(t,x,y)$ using Chapman-Kolmogorov identity as follows
\begin{align*}
k_B(t,x,y)&=\int_{B_n\left(\frac{x+y}2,R\)}+\int_{B\bs B_n\left(\frac{x+y}2,R\right)}k_B(t/2,x,z)k_B(t/2,z,y)dz\\[8pt]
&=I_1+I_2,
\end{align*}
where $R=R(x,y)=\sqrt{\delta\left(\frac{x+y}2\right)\sqrt t}$. Such a choice of $R(x,y)$ guarantees that $R(x,y)/\sqrt t\rightarrow\infty$ and $R(x,y)/\delta\left(\frac{x+y}2\right)\rightarrow0$ as $\delta\left(\frac{x+y}2\right)/\sqrt t\rightarrow\infty$. The latter limit implies    $\delta\left(z\right)\approx \delta\left(\frac{x+y}2\right)$  for $z\in B_n\left(\frac{x+y}2,R\right)$. 
Hence, by Lemma \ref{lem:1.1}, we get for $\delta\left(\frac{x+y}2\right)/\sqrt t$ large enough
\begin{align*}\nonumber
&\left|I_1-\int_{B_n\left(\frac{x+y}2,R\)}k_{H_x}(t/2,x,z)k_{H_y}(t/2,z,y)dz\right| \\\nonumber
&\leq\left|\int_{B_n\left(\frac{x+y}2,R\)}k_{B}(t/2,x,z)\(k_{B}(t/2,z,y)-k_{H_y}(t/2,z,y)\)dz\right|\\\nonumber
&\ \ \ +\left|\int_{B_n\left(\frac{x+y}2,R\)}\(k_{B}(t/2,x,z)-k_{H_x}(t,x,z)\)k_{H_y}(t/2,z,y)dz\right|\\\nonumber
&\lesssim \frac{t}{\(\delta\left(\frac{x+y}2\right)\)^2}\int_{B_n\left(\frac{x+y}2,R\)}k_B(t/2,x,z)k_B(t/2,z,y)dz\\
&\leq \frac{t}{\(\delta\left(\frac{x+y}2\right)\)^2}k_B(t,x,y).
\end{align*}
Note that for large values of $\delta\left(\frac{x+y}2\right)/\sqrt t$, the right-hand side term in \eqref{eq:htxy:est}  is dominating, and therefore, by inequalities $|x-y|^2\leq 8\delta\left(\frac{x+y}2\right)$ and $\delta\left(w\right)\leq \delta_{H_x}(w)\wedge \delta_{H_y}(w)$, $w\in B$, we have
\begin{align*}
k_B(t,x,y)&\lesssim \(1\wedge \frac{\delta(x)\,\delta_{H_x}\left(\frac{x+y}2\right)}{t}\)\(1\wedge \frac{\delta(y)\,\delta_{H_y}\left(\frac{x+y}2\right)}{t}\)k(t,x,y)\\
&\approx \left(1-e^{-2\delta(x)\,\delta_{H_x}\left(\frac{x+y}2\right)/t}\right)\left(1-e^{-2\delta(y)\,\delta_{H_y}\left(\frac{x+y}2\right)/t}\right)k(t,x,y),
\end{align*}

which yields
\begin{align}\nonumber
&\left|I_1-\int_{B_n\left(\frac{x+y}2,R\)}k_{H_x}(t/2,x,z)k_{H_y}(t/2,z,y)dz\right|\\\label{eq:J2a}
&\lesssim \frac{t}{\(\delta\left(\frac{x+y}2\right)\)^2}\left(1-e^{-2\delta(x)\,\delta_{H_x}\left(\frac{x+y}2\right)/t}\right)\left(1-e^{-2\delta(y)\,\delta_{H_y}\left(\frac{x+y}2\right)/t}\right)k(t,x,y).
\end{align}
Next, since $\left|\delta_{H_x}(z)-\delta_{H_x}\(\frac{x+y}2\)\right|\leq \left|z-\frac{x+y}2\right|<R$, \eqref{eq:1-e} gives us
\begin{align*}
\left|\frac{1-e^{-2\delta(x)\,\delta_{H_x}\left(z\right)/t}}{1-e^{-2\delta(x)\,\delta_{H_x}\left(\frac{x+y}{2}\right)/t}}-1\right|
&\lesssim \left|\frac{\delta_{H_x}\left(z\right)-\delta_{H_x}\(\frac{x+y}{2}\)}{\delta_{H_x}\(\frac{x+y}{2}\)}\right|\leq \sqrt{\frac{\sqrt t}{\delta\(\frac{x+y}2\)}}, 
\end{align*}
which, together with \eqref{eq:kHform}, follows
\begin{align*}
&\left|k_{H_x}(t/2,x,z)-\left(1-e^{-2\delta(x)\,\delta_{H_x}\left(\frac{x+y}2\right)/t}\right)k(t/2,x,z)\right|\\
&\lesssim \sqrt{\frac{\sqrt t}{\delta\(\frac{x+y}2\)}}\left(1-e^{-2\delta(x)\,\delta_{H_x}\left(\frac{x+y}2\right)/t}\right)k(t/2,x,z).
\end{align*}
The same bound holds if we switch $x$ and $y$. Thus, analogously as in \eqref{eq:aux7}, we obtain
\begin{align*}
&\Bigg|\int_{B_n\left(\frac{x+y}2,R\)}k_{H_x}(t/2,x,z)k_{H_y}(t/2,z,y)dz\\
&\ \ \ \ -\left(1-e^{-2\delta(x)\,\delta_{H_x}\left(\frac{x+y}2\right)/t}\right)\left(1-e^{-2\delta(y)\,\delta_{H_y}\left(\frac{x+y}2\right)/t}\right)k(t,x,y)\Bigg|\\
&\leq\Bigg|\int_{B_n\left(\frac{x+y}2,R\)}k_{H_x}(t/2,x,z)k_{H_y}(t/2,z,y)\\
&\ \ \ \ -\left(1-e^{-2\delta(x)\,\delta_{H_x}\left(\frac{x+y}2\right)/t}\right)\left(1-e^{-2\delta(y)\,\delta_{H_y}\left(\frac{x+y}2\right)/t}\right)k(t/2,x,z)k(t/2,z,y)dz\Bigg|\\
&\ \ \ +\left(1-e^{-2\delta(x)\,\delta_{H_x}\left(\frac{x+y}2\right)/t}\right)\left(1-e^{-2\delta(y)\,\delta_{H_y}\left(\frac{x+y}2\right)/t}\right)\int_{\(B_n\left(\frac{x+y}2,R\)\)^c}k(t/2,x,z)k(t/2,z,y)|dz\\
&\lesssim \left(1-e^{-2\delta(x)\,\delta_{H_x}\left(\frac{x+y}2\right)/t}\right)\left(1-e^{-2\delta(y)\,\delta_{H_y}\left(\frac{x+y}2\right)/t}\right)\\
&\ \ \ \times\(\sqrt{\frac{\sqrt t}{\delta\(\frac{x+y}2\)}}\int_{B_n\left(\frac{x+y}2,R\)}k(t/2,x,z)k(t/2,z,y)dz+\int_{\(B_n\left(\frac{x+y}2,R\)\)^c}k(t/2,x,z)k(t/2,z,y)dz\)\\
&\lesssim \left(1-e^{-2\delta(x)\,\delta_{H_x}\left(\frac{x+y}2\right)/t}\right)\left(1-e^{-2\delta(y)\,\delta_{H_y}\left(\frac{x+y}2\right)/t}\right)k(t,x,y)\(\sqrt{\frac{\sqrt t}{\delta\(\frac{x+y}2\)}}+e^{-R^2/2t}\)\\
&\lesssim \sqrt{\frac{\sqrt t}{\delta\(\frac{x+y}2\)}}\left(1-e^{-2\delta(x)\,\delta_{H_x}\left(\frac{x+y}2\right)/t}\right)\left(1-e^{-2\delta(y)\,\delta_{H_y}\left(\frac{x+y}2\right)/t}\right)k(t,x,y),
\end{align*}
where we also used Corollary \ref{cor:CKRR}. Combining this with \eqref{eq:J2a} we arrive at
\begin{align}\nonumber
&\Bigg|I_1 -\left(1-e^{-2\delta(x)\,\delta_{H_x}\left(\frac{x+y}2\right)/t}\right)\left(1-e^{-2\delta(y)\,\delta_{H_y}\left(\frac{x+y}2\right)/t}\right)k(t,x,y)\Bigg|\\\label{eq:aux4}
&\lesssim \sqrt{\frac{\sqrt t}{\delta\(\frac{x+y}2\)}}\left(1-e^{-2\delta(x)\,\delta_{H_x}\left(\frac{x+y}2\right)/t}\right)\left(1-e^{-2\delta(y)\,\delta_{H_y}\left(\frac{x+y}2\right)/t}\right)k(t,x,y).
\end{align}

 It is now enough to  show   that $I_2$ is  suitably small.

 If $\angle(x,y)>\pi/2$ (where by $\angle(x,y)$ we mean the smaller non-negative angle between vectors $\vec x=(0,x)$ and $\vec y=(0,y)$), we have $\delta_{H_x}\left(\frac{x+y}2\right)\approx \delta_{H_y}\left(\frac{x+y}2\right)\approx1$. Consequently, using  \eqref{eq:kHest} and   Corollary \ref{cor:CKRR}, we get
\begin{align}\nonumber
I_2&\lesssim
\int_{B\bs B_n\left(\frac{x+y}2,R\right)}k_{H_x}(t/2,x,z)k_{H_y}(t/2,z,y)dz\\\nonumber
&\lesssim\left(1\wedge\frac{\delta(x)}{t}\right)\left(1\wedge\frac{\delta(y)}{t}\right)
\int_{B\bs B_n\left(\frac{x+y}2,R\right)}k(t/2,x,z)k(t/2,z,y)dz\\\label{eq:bigangle}
&\lesssim\left(1\wedge\frac{\delta(x)\delta_{H_x}\left(\frac{x+y}2\right)}{t}\right)\left(1\wedge\frac{\delta(y)\delta_{H_y}\left(\frac{x+y}2\right)}{t}\right)k(t,x,y)e^{-\delta\left(\frac{x+y}2\right)/2\sqrt t}.
\end{align}
Consider now $\angle(x,y)\leq\pi/2$. The main consequence of this assumption is that $\delta(x)\approx \delta_{H_{xy}}(x)$ and $\delta(y)\approx \delta_{H_{xy}}(y)$, where $H_{xy}$ is defined in Section \ref{sec:notation}.

Furthermore, recalling that $P_{xy}:=\partial H_{xy}$, simple geometry  and (\ref{eq:x_0-y_0}) yield 
\begin{align}\nonumber
\rho(x,y)&:=\max\{dist(z,P_{xy}):z\in B\cap (H_{x,y})^c\}\\\nonumber
&\,\,=1-\sqrt{1-\frac14\left|\frac{x}{|x|}-\frac{y}{|y|}\right|^2}\leq \frac14\left|\frac{x}{|x|}-\frac{y}{|y|}\right|^2\\
\label{eq:rho}
&\,\,<6\,\delta\left(\frac{x+y}2\right).
\end{align}
Note that for any $z\in B$  
$$\delta(z)\leq dist\(z,P_{xy}\)+\rho(x,y),$$
where $dist\(z,P_{xy}\)$ denotes the distance from $z$ to $P_{xy}$. Obiously, for $z\in H_{xy}$ we have $dist\(z,P_{xy}\)=\delta_{H_{xy}}(z)$. Then, we split $I_2$ into another two integrals:
\begin{align*}
I_2&= \int_{\substack{z\in B\bs B_n\left(\frac{x+y}2,\,R\right)\\[5pt] dist(z,P_{xy})\leq6\delta\(\frac{x+y}{2}\)}}+\int_{\substack{z\in B\bs B_n\left(\frac{x+y}2,\,R\right)\\[5pt]dist(z,P_{xy})>6\delta\(\frac{x+y}{2}\)}}k_B(t/2,x,z)k_B(t/2,z,y)dz\\[4pt]
&=:I_{2,1}+I_{2,2}.
\end{align*}
Since $dist(x,P_{xy})\leq \delta(x)\leq 2\delta\(\frac{x+y}2\)$, then the first inequality in \eqref{eq:parallel} implies that for $z\in B$  satisfying $dist(z,P_{xy})\leq6\delta\left(\frac{x+y}2\right)$ it holds
\begin{align*}
|x-z|^2&\leq 8\delta\left(\frac{x+z}2\right)\leq 8\(dist\(\frac{x+z}2,P_{xy}\)+\rho(x,y)\)\\
&=8\(\frac12dist\(x,P_{xy}\)+\frac12dist\(z,P_{xy}\)+\rho(x,y)\)\lesssim \delta\left(\frac{x+y}2\right).
\end{align*}
 Applying this and $\delta(z)\leq dist\(z,P_{xy}\)+\rho(x,y)<24\delta\(\frac{x+y}{2}\)$ to upper bound in  \eqref{eq:MS}  we obtain
$$k_B(t/2,x,z)\lesssim \left(1\wedge\frac{\delta(x)\delta\left(\frac{x+y}2\right)}{t}\right)k(t/2,x,z).$$
Hence, by Corollary \ref{cor:CKRR},
\begin{align}\nonumber
I_{2,1}
&\lesssim\left(1\wedge\frac{\delta(x)\delta\left(\frac{x+y}2\right)}{t}\right)\left(1\wedge\frac{\delta(y)\delta\left(\frac{x+y}2\right)}{t}\right) \int_{\left(B_n\left(\frac{x+y}2,\,R\right)\right)^c}k(t/2,x,z)k(t/2,z,y)dz\\\label{eq:J21}
&\lesssim\left(1\wedge\frac{\delta(x)\delta_{H_x}\left(\frac{x+y}2\right)}{t}\right)\left(1\wedge\frac{\delta(y)\delta_{H_y}\left(\frac{x+y}2\right)}{t}\right)k(t,x,y)e^{-\delta\left(\frac{x+y}2\right)/2\sqrt t}.
\end{align}
On the other hand, for  any $z\in B$ such that $dist(z,P_{xy})>6\,\delta\left(\frac{x+y}2\right)$, we have
\begin{align}\label{eq:aux12}
\delta(z)\leq dist(z,P_{xy})+\rho(x,y)\leq 2dist(z,P_{xy})=2 \delta_{H_{xy}}\left(z\right)
\end{align}
and, by \eqref{eq:parallel} and $\delta_{H_{xy}}(z)>\delta_{H_{xy}}(x)$,
\begin{align*}
|x-z|^2&\leq8\delta\left(\frac{x+z}2\right)\leq 8\(dist\(\frac{x+z}2,P_{xy}\)+\rho(x,y)\)\\
&\leq 8\(dist\(z,P_{xy}\)+\rho(x,y)\)\lesssim\delta_{H_{xy}}\left(z\right).
\end{align*}
Consequently, by upper bound in \eqref{eq:MS} and the estimate $\delta(x)\approx \delta_{H_{xy}}(x)$, we get
\begin{align*}
k(t/2,x,z)&\lesssim \left[\left(1\wedge\frac{\delta(x)\delta(z)}t\right)+\left(1\wedge\frac{\delta(x)\delta_{H_{xy}}(z)}t\right)\right]k(t/2,x,z)\\
&\lesssim \left(1\wedge\frac{\delta_{H_{xy}}(x)\delta_{H_{xy}}(z)}t\right)k(t/2,x,z)\approx k_{H_{xy}}(t/2,x,z),
\end{align*}
and the same inequality holds with $y$ instead of $x$. Thus, by Corollary \ref{cor:CKHH}, we obtain 
\begin{align*}
I_{2,2}&\lesssim \int_{H_{xy}\bs B_n\left(\frac{x+y}2,\,R(x,y)\right)}k_{H_{xy}}(t/2,x,z)k_{H_{xy}}(t/2,z,y)dz\\
&\lesssim k_{H_{xy}}(t,x,y)e^{-\delta\(\frac{x+y}2\)/2\sqrt t}\lesssim \(1\wedge\frac{\delta_{H_{xy}}(x)\delta_{H_{xy}}(y)}{t}\)k(t,x,y)e^{-\delta\(\frac{x+y}2\)/2\sqrt t}.
\end{align*}
Hence, since $\delta_{H_x}\(\frac{x+y}2\), \delta_{H_y}\(\frac{x+y}2\)\geq \delta\(\frac{x+y}2\)$ we may bound for $t<1$ and $\delta\(\frac{x+y}{t}\)/\sqrt t>1$ as follows
\begin{align}\nonumber
I_{2,2}&\lesssim \(1\wedge\frac{\delta_{H_{xy}}(x)}{\sqrt t}\)\(1\wedge\frac{\delta_{H_{xy}}(y)}{\sqrt t}\)k(t,x,y)e^{-\delta\(\frac{x+y}2\)/2\sqrt t}\\\label{eq:J22}
&\leq \(1\wedge\frac{\delta(x)\delta_{H_x}\(\frac{x+y}2\)}{t}\)\(1\wedge\frac{\delta(y)\delta_{H_y}\(\frac{x+y}2\)}{t}\)k(t,x,y)e^{-\delta\(\frac{x+y}2\)/2\sqrt t}.
\end{align}
Finally, combining \eqref{eq:aux4}, \eqref{eq:bigangle}, \eqref{eq:J21} and \eqref{eq:J22} we get
\begin{align*}\nonumber
&\Bigg|k_B(t,x,y) -\left(1-e^{-2\delta(x)\,\delta_{H_x}\left(\frac{x+y}2\right)/t}\right)\left(1-e^{-2\delta(y)\,\delta_{H_y}\left(\frac{x+y}2\right)/t}\right)k(t,x,y)\Bigg|\\
&\lesssim \sqrt{\frac{\sqrt t}{\delta\(\frac{x+y}2\)}}\left(1-e^{-2\delta(x)\,\delta_{H_x}\left(\frac{x+y}2\right)/t}\right)\left(1-e^{-2\delta(y)\,\delta_{H_y}\left(\frac{x+y}2\right)/t}\right)k(t,x,y),
\end{align*}
which is equivalent to the assertion of the theorem.
\end{proof}

\section{Proof of Theorem  \ref{thm:Thm2}}
 We start this section with another bound for difference   between  $k_B(t,x,y)$ and $k_{H_x}(t,x,y)$. The below-given lemma implies assertion of Theorem  \ref{thm:Thm2} with additional assumptions  that $\delta(y)$ tends to $\delta_{H_x}(y)$ and $\delta(x)<\delta(y)$, which seem to be quite natural when comparing  $k_B(t,x,y)$ with $k_{H_x}(t,x,y)$ for $x$ and $y$ close to each other and to the boundary of $B$.
\begin{lemma}\label{lem:2.1}
Let $t<1$ and $x,y\in B$ with $\delta(x)\leq\delta(y)$. Then
\begin{align*}
|k_B(t,x,y)-k_{H_x}(t,x,y)|&\lesssim \frac{\delta(x)\delta(y)}{t}k(t,x,y)\(\sqrt t+\frac{|x-y|^{2}}{\sqrt t}+\frac{\delta_{H_x}(y)-\delta(y)}{\delta(y)}\).
\end{align*}
\end{lemma}

\begin{proof}
First, let us note that if $\delta_{H_x}(y)-\delta(y)\geq \frac1{100}\delta(y)$, then $\delta_{H_x}(y)\leq 101\(\delta_{H_x}(y)-\delta(y)\)$ and the assertion is obvious since, by \eqref{eq:kHest},
$$k_B(t,x,y)\leq k_{H_x}(t,x,y)\lesssim \frac{\delta(x)\delta_{H_x}(y)}{t}k(t,x,y)\leq 101\frac{\delta(x)\delta(y)}{t}k(t,x,y)\frac{\delta_{H_x}(y)-\delta(y)}{\delta(y)}.$$
Similarly, if $\delta_{H_x}(y)-\delta(y)< \frac1{100}\delta(y)$, then $\delta_{H_x}(y)<\frac{101}{100}\delta(y)$ and the assertion is clear for $|x-y|^2/\sqrt t\geq \frac14$ by
$$k_B(t,x,y)\leq k_{H_x}(t,x,y)\lesssim\frac{\delta(x)\delta_{H_x}(y)}{t}k(t,x,y)\leq \frac{101}{25}\frac{\delta(x)\delta(y)}{t}k(t,x,y)\frac{|x-y|^2}{\sqrt t} .$$
For these reasons, throughout the proof we assume
\begin{align}\label{eq:assumpinlemma}
\delta_{H_x}(y)-\delta(y)< \frac1{100}\delta(y),\hspace{10mm}\text{ and }\hspace{13mm} \frac{|x-y|^2}{\sqrt t}<\frac14.
\end{align}
Without loss of generality we also assume that $x=(0,...,0,x_n)$, $x_n>0$, and $y=(0,...,0,y_{n-1},y_n)$, $y_{n-1},y_n\geq0$. In particular,  this implies $H_x=\{z\in\R^n:z_n<1\}$. By Strong Markov property (or Hunt formula) we conclude
$$k_{H_{x}}(t,x,y)=k_B(t,x,y)+r(t,x,y),$$
where
\begin{align*}
r(t,x,y)&=\E^{y}\[\tau_{B}<t; k_{H_{x}}\(t-\tau_{B},W\(\tau_{B}\),x\)\]\\
&=\E^{y}\[\tau_{B}<t;W_n\(\tau_{B}\)<\frac12 ; k_{H_{x}}\(t-\tau_{B},W\(\tau_{B}\),x\)\]\\
&\ \ \ +\E^{y}\[\tau_{B}<t;W_n\(\tau_{B}\)\geq\frac12 ; k_{H_{x}}\(t-\tau_{B},W\(\tau_{B}\),x\)\]\\
&:=r_1(t,x,y)+r_2(t,x,y).
\end{align*}
Our aim is therefore to estimate $r(t,x,y)$. We start with  $r_1(t,x,y)$. First, let us note that for $x=(0,...,0,x_n)$ with $x_n\geq 0$ and $z\in\partial B$ such that $z_n<1/2$   it holds $|x-z|\geq \sqrt 2/2$.  Then, using the assumption $|x-y|^2\leq \frac14\sqrt t<\frac14$ and  the fact that the function $e^{-1/8s}/s^{n/2+1}$ may be estimated by an increasing function on the interval $(0,1)$,  we get for $s<t<1$ 
\begin{align*}
k_{H_{x}}\(s,z,x\)&\lesssim \frac{\delta(x)}{s^{1+n/2}}e^{-1/8s}\lesssim \frac{\delta(x)}{t^{1+n/2}}e^{-1/8t}\leq \frac{\delta(x)}t\frac{e^{-|x-y|^2/4t}}{t^{n/2}}e^{-1/16t}\lesssim{\delta(x)}k(t,y,x)e^{-1/16t},
\end{align*}
and consequently
\begin{align}\nonumber
r_1(t,x,y)&\lesssim \frac{\delta(x)}tk(t,x,y)e^{-1/16t}\E^{y}\[W_n\(\tau_{B}\)<\frac12 \]\\\label{eq:r_1}
&\lesssim \frac{\delta(y)\delta(x)}tk(t,x,y)e^{-1/16t}\lesssim \frac{\delta(y)\delta(x)}tk(t,x,y)\sqrt t,
\end{align}
where we used  the estimate $\[W_1\(\tau_{B}\)<\frac12 \]\lesssim \delta(x)$, which   follows e.g. from the formula for the Poisson kernel of a ball.   

Let us now deal with $r_2(t,x,y)$. Denoting $\tilde z=(z_1,z_2,...,z_{n-1})\in\R^{n-1}$ we have for $z\in \partial B\cap\{z_n\geq0\}$
$$\delta_{H_x}(z)=1-\sqrt{1-|\tilde z|^2}=\frac{|\tilde z|^2}{1+\sqrt{1-|\tilde z|^2}}\leq |\tilde z|^2.$$
Hence, using \eqref{eq:qest} and  \eqref{eq:kHest}, we obtain
\begin{align*}
r_2(t,x,y)&=\int_{z\in\partial B: z_n\geq\frac12}\int_0^t q_y(s,z) k_{H}(t-s,z,x)ds\,dz\\
&\lesssim\int_{z\in\partial B:z_n\geq\frac12}\int_0^t \frac{\delta(y)}s\(1+\frac{|y-z|^4}s\)k(t,y,z)\frac{\delta(x)\delta_{H_x}(z)}{t-s} k(t-s,z,x)ds\,dz.\\
&\lesssim\delta(x)\delta(y)\int_{z\in\partial B:z_n\geq\frac12}|\tilde z|^2\int_0^t \(1+\frac{|y-z|^4}s\)\frac{\exp\(-\frac{|y-z|^2}{4s}-\frac{|x-z|^2}{4(t-s)}\)}{s^{\frac n2+1}(t-s)^{\frac n2+1}} ds\,dz.
\end{align*}
Splitting the inner integral into two by separating summands in  the factor $\(1+\frac{|y-z|^4}s\)$, and then applying Proposition \ref{prop:estints} with $\alpha=\beta=\frac n2+1$ and $\alpha=\frac n2+2$, $\beta=\frac n2+1$, respectively, we get
\begin{align*}
&r_2(t,x,y)\lesssim \\
&\int_{z\in\partial B:z_n\geq\frac12}\delta(x)\delta(y)| \tilde z|^2e^{-\frac{(|x-z|+|z-y|)^2}{4t}}\Bigg(\frac{\(\frac t{|y-z|^2}\)^{n/2}(1+|y-z|^2)+\(\frac t{|z-x|^2}\)^{n/2}\(1+\frac{|y-z|^4}t\)}{t^{n+1}}\\
&\ \ \  +\frac1{t^{n+1}}\frac{(|x-z|+|y-z|)^{n}}{|x-z|^{n/2}|y-z|^{n/2}}\(1+\frac{(|x-z|+|z-y|)|y-z|^3}{t}\)\sqrt{\frac{t}{t+|x-z||y-z|}}\Bigg)\,dz\\
&=\int_{A_1}...dz +\int_{A_2}...dz +\int_{A_3}...dz=:I_1+I_2+I_3, 
\end{align*}
where,
\begin{align*}
A_1&=\{z\in \partial B: |z-x|\leq 2|x-y|, |z-y|>\frac18|x-y|\},\\
A_2&=\{z\in \partial B: |z-y|\leq\frac18|x-y|\},\\ 
A_3&=\{z\in\partial B:z_n\geq1/2, |z-x|>2|x-y|\}.
\end{align*}
It is clear, that for $z\in A_3$ we have $|x-z|\approx |y-z|$ and $(|x-z|+|z-y|)^2\geq |x-z|^2+|z-y|^2\geq |x-z|^2+|x-y|^2$, thus
$$
I_3\lesssim\delta(x)\delta(y)\frac{e^{-|x-y|^2/4t}}{t^{n/2+1}}\int_{A_3}| \tilde z|^2e^{{-|x-z|^2/4t}}\(1+\frac{|x-z|^4}{t}\)\Bigg(\frac{1}{|x-z|^{n}} +\frac{1}{t^{n/2}}\Bigg)\,dz.$$
Furthermore, we estimate $|x-z|\geq  |\tilde z|$ and
$$e^{{-|x-z|^2/4t}}\(1+\frac{|x-z|^4}{t}\)\leq e^{{-|x-z|^2/8t}}\(1+|x-z|^2\sup_{v\in(0,\infty)}\{ve^{v/8}\}\)\lesssim e^{{-|x-z|^2/8t}},$$
 which yields
$$
I_3\lesssim\delta(x)\delta(y)\frac{e^{-|x-y|^2/4t}}{t^{n/2+1}}\int_{A_3}e^{{-|\tilde z|^2/8t}}\Bigg(\frac{1}{|\tilde z|^{n-2}} +\frac{|\tilde z|^2}{t^{n/2}}\Bigg)\,dz.$$
Next, we parametrize $A_3$ by $B_{n-1}\({0,\sqrt3/2}\)\ni v\longrightarrow (v,\sqrt{1-|v|^2})$. Since surface element of this mapping is bounded, we obtain
\begin{align}\nonumber
I_3&\lesssim \delta(x)\delta(y)\frac{e^{-|x-y|^2/4t}}{t^{n/2+1}}\int_{B_{n-1}\({0,\sqrt3/2}\)}e^{{-|v|^2/8t}}\Bigg(\frac{1}{|v|^{n-2}} +\frac{|v|^2}{t^{n/2}}\Bigg)\,dv\\\nonumber
&=\sqrt  t \, \delta(x)\delta(y)\frac{e^{-|x-y|^2/4t}}{t^{n/2+1}}\int_{B_{n-1}\({0,\sqrt3/2\sqrt t}\)}e^{{-|w|^2/8}}\Bigg(\frac{1}{|w|^{n-2}} +|w|^2\Bigg)\,dw\\\label{eq:I3}
&\approx \sqrt t\,\frac{\delta(x)\delta(y)}{t}k(t,x,y).
\end{align}
Let us pass to estimating  integrals $I_1$ and $I_2$. The latter assumption in \eqref{eq:assumpinlemma} gives us for $z\in A_1\cup A_2$ 
$$1+\frac{|y-z|^4}{t}, 1+\frac{(|x-z|+|z-y|)|y-z|^3}{t}\lesssim1+\frac{|x-y|^4}{t}\approx 1.$$
Furthermore, let $z'$ be an orthogonal projection of $z$ onto the line containing $x$ and $y$. Then, we have
\begin{align*}
(|x-z|+|y-z|)^2&=|x-z'|^2+|y-z'|^2+2|z-z'|^2+2|x-z||y-z|\\
&=(|x-z'|+|y-z'|)^2-2|x-z'||y-z'|+2|x-z||y-z|+2|z-z'|^2\\
&\geq |x-y|^2+2\(|x-z||y-z|-|x-z'||y-z'|\)\\
&= |x-y|^2+2\(\sqrt{|x-z'|^2+|z-z'|^2}\sqrt{|y-z'|^2+|z-z'|^2}-|x-z'||y-z'|\)\\
&= |x-y|^2+2\frac{|z-z'|^2\(|x-z'|^2+|y-z'|^2\)+|z-z'|^4}{\sqrt{|x-z'|^2+|z-z'|^2}\sqrt{|y-z'|^2+|z-z'|^2}+|x-z'||y-z'|}\\
&\geq |x-y|^2+\frac{|z-z'|^2\(|x-z'|^2+|y-z'|^2\)}{\sqrt{|x-z'|^2+|z-z'|^2}\sqrt{|y-z'|^2+|z-z'|^2}}\\
&\geq |x-y|^2+\frac14\frac{|z-z'|^2|x-y|^2}{|x-z||y-z|}.
\end{align*} 
Hence, both of the integrands in $I_1$ and $I_2$ are bounded (up to a multiplicative constant) by
\begin{align}\label{eq:integrand12}
&\frac{\delta(x)\delta(y)}{t}k(t,x,y)\exp\({-\frac1{16t}\frac{|z-z'|^2|x-y|^2}{|x-z||y-z|}}\)\\\nonumber
&\times| \tilde z|^2\Bigg(\frac 1{(|x-z|\wedge|y-z|)^n} +\(\frac{|x-y|^2}{t}\)^{n/2}\frac{\sqrt t}{\(|y-z||z-x|\)^{(n+1)/2}}\Bigg),
\end{align}
where we also bounded $\sqrt{t+|x-z||y-z|}>\sqrt{|x-z||y-z|)}$.

For $z\in A_1$ we have $|y-z|\approx |x-y|$. Furthermore, due to the assumed form of $x$ and $y$, we have  $z'=(0,...,0,z_{n-1}, z_n)$, and 
therefore $|z-z'|\geq|(z_1,...,z_{n-2})|$. Hence, estimating additionally $|x-z|\approx |\tilde z|+\delta(x)$, we get for $z\in A_1$
\begin{align}\label{eq:fractions}
\frac{|z-z'|^2|x-y|^2}{|x-z||y-z|}\gtrsim\frac{|z-z'|^2|x-y|}{|x-z|}\gtrsim \frac{|(z_1,...,z_{n-2})|^2|x-y|}{|(z_1,...,z_{n-2})|+|z_{n-1}|+\delta(x)}.
\end{align}
All together  leads to
\begin{align*}
I_1\lesssim \frac{\delta(x)\delta(y)}{t}k(t,x,y)\int_{A_1}&e^{-\frac{c}{t}\frac{|(z_1,...,z_{n-2})|^2|x-y|}{|(z_1,...,z_{n-2})|+|z_{n-1}|+\delta(x)}}\Bigg(\frac1{|\tilde z|^{n-2}} +\(\frac{|x-y|}t\)^{(n-1)/2}\frac{1}{|\tilde z|^{(n-3)/2}}\Bigg)dz,
\end{align*}
for some constant $c>0$. Next, we parametrize $A_1$ by
 $$ S_1\ni (v,w)\longrightarrow (v,w,\sqrt{1-|(v,w)|^2})\in A_1,$$
for some $S_1\subset B_{n-2}(0,2|x-y|)\times (-2|x-y|,2|x-y|)$. This gives us
 \begin{align}\label{eq:I1.1}
 I_1&\lesssim \frac{\delta(x)\delta(y)}{t}k(t,x,y)\int_{B_{n-2}(0,2|x-y|)}\\\nonumber
&\ \ \ \times\int_{-2|x-y|}^{2|x-y|}\frac1{(|v|+|w|)^{n-2}} +\(\frac{|x-y|}t\)^{(n-1)/2}\frac{e^{-\frac ct\frac{|v|^2|x-y|}{|v|+|w|+\delta(x)}}}{(|v|+|w|+\delta(x))^{(n-3)/2}}dw\,dv,
\end{align}
where we also omitted the exponent  in the first fraction. We deal with the summands of the integrand separately. First, 
 \begin{align}\label{eq:I11}
 &\int_{B_{n-2}(0,2|x-y|)}\int_{-2|x-y|}^{2|x-y|}\frac{dw\,dv}{(|v|+|w|)^{n-2}}\leq \int_{B_{n-1}(0,4|x-y|)}\frac1{|z|^{n-2}}dz\approx |x-y|.
\end{align}
In case of the latter integral we consider dimension $n=2$ separately. Precisely,  for $n=2$ we obtain just a single integral 
\begin{align*}
\(\frac{|x-y|}t\)^{1/2}\int_{-2|x-y|}^{2|x-y|}{|w|^{1/2}}dw\,dv\lesssim \frac{|x-y|^2}{\sqrt t}.
\end{align*} 
For $n\geq 3$ we will use the inequality $\delta(x)\leq 2|x-y|$, which is true  since else $A_1\cup A_2=\phi$ and $I_1=I_2=0$. Then, for $|v|\leq 2|x-y|$
 \begin{align}\nonumber
  \int_{-2|x-y|}^{2|x-y|}\frac{e^{-\frac ct\frac{|v|^2|x-y|}{|v|+|w|+\delta(x)}}}{(|v|+|w|+\delta(x))^{(n-3)/2}}dw
  &=2\int_{|v|+\delta(x)}^{|v|+2|x-y|+\delta(x)}\frac{e^{-\frac ct\frac{|v|^2|x-y|}{w}}}{w^{(n-3)/2}}dw\leq 2\int_{|v|}^{6|x-y|}\frac{e^{-\frac ct\frac{|v|^2|x-y|}{w}}}{w^{(n-3)/2}}dw\\\nonumber
   &=2\(\frac{t}{|v|^2|x-y|}\)^{(n-5)/2}\int^{|v||x-y|/t}_{|v|^2/6t}{e^{-c\,u}}u^{(n-7)/2}du\\\nonumber
    &\leq 2\(\frac{t}{|v|^2|x-y|}\)^{(n-5)/2}e^{-c|v|^2/6t}\int_{0}^{\infty}{e^{-c\,u}}\(u+\frac{|v|^2}{6t}\)^{(n-7)/2}du\\\nonumber
    &\lesssim \(\frac{t}{|v|^2|x-y|}\)^{(n-5)/2}e^{-c|v|^2/6t}\frac{\(\frac{|v|^2}{t}\)^{(n-5)/2}}{1\vee \frac{|v|^2}{t}}\\\label{eq:intdw}
&\approx  \frac{e^{-c|v|^2/6t}}{|x-y|^{(n-5)/2}\(1+\frac{|v|^2}{t}\)}.
\end{align}
This follows
\begin{align*}
&\(\frac{|x-y|}t\)^{(n-1)/2}\int_{B_{n-2}(0,2|x-y|)}\int_{-2|x-y|}^{2|x-y|}\frac{e^{-\frac ct\frac{|v|^2|x-y|}{|v|+|w|+\delta(x)}}}{(|v|+|w|)^{(n-3)/2}}dw\,dv\\
&\lesssim \frac{|x-y|^2}{t^{(n-1)/2}}\int_{\R^{n-2}}\frac{e^{-c|v|^2/6t}}{1+\(\frac{|v|^2}{t}\)}\,dv=\frac{|x-y|^2}{\sqrt t}\int_{\R^{n-2}}\frac{e^{-c|v|^2/6}}{1+{|v|^2}}\,dv\\
&\approx \frac{|x-y|^{2}}{\sqrt t},
\end{align*}
which is the same bound as in the case $n=2$. 
Applying this and  \eqref{eq:I11}  to \eqref{eq:I1.1}, we  arrive at
\begin{align}\label{eq:I1}
I_1\lesssim \frac{\delta(x)\delta(y)}{t}k(t,x,y)\(|x-y|+\frac{|x-y|^{2}}{\sqrt t}\).
\end{align}
Estimation of $I_2$ is very similar to estimation of $I_1$ at many points. Nevertheless, some additional preparation is needed. If $\delta(y)\geq |x-y|/8$, then clearly $A_2=\phi$ and consequently $I_2=0$, so we assume $\delta(y)< |x-y|/8$.  Furthermore,  it happens that 
\begin{align}\label{eq:tildeyvs|x-y|}
\frac12|x-y|\leq |\tilde y|=y_{n-1}\leq|x-y|.
\end{align}
The latter inequality follows directly from the form of $x$ and $y$. To see the first one, assume it is not true i.e. $|\tilde y|<|x-y|/2$. Then, due to the inequality $y_n\leq x_n$, that follows from the assumption $\delta(x)\leq \delta(y)$, we get
$$y_n= x_n-\sqrt{|x-y|^2-|\tilde y|^2}\leq 1-\frac{\sqrt 3}2|x-y|,$$
and hence
$$\delta(y)\geq \sqrt{1-|\tilde y|^2}-y_n\geq \sqrt{1-\frac14|x-y|^2}-1+\frac{\sqrt 3}2|x-y|\geq -\frac14|x-y|^2+\frac{\sqrt 3}2|x-y|\geq \frac{\sqrt 3}4|x-y|,$$
which contradicts the assumption $\delta(x)<|x-y|/8$. Therefore, from \eqref{eq:tildeyvs|x-y|} we have
\begin{align}\label{eq:aux11}
\delta_{H_x}(y)-\delta(y)\geq 1-\sqrt{1-|\tilde y|^2}\geq \frac12|\tilde y|^2\geq \frac18|x-y|^2,
\end{align}
which, in view of the assumption $\delta_{H_x}(y)-\delta(y)\leq\frac1{100} \delta(y)$, gives
\begin{align}\label{eq:|x-y|2<delta(y)}
|x-y|^2\leq \frac8{100}\delta(y).
\end{align}
Next, we will show that 
\begin{align}\label{eq:z-z'}
|z-z'|\geq \frac12|(z_1,...,z_{n-2})|+\frac1{16}\delta(y).
\end{align}
From the definition of $A_2$, \eqref{eq:tildeyvs|x-y|} and the form of $y$,  we have for $z\in A_2$
\begin{align}\label{eq:z<deltay}
\frac38|x-y|<|\tilde z|<\frac98|x-y|,
\end{align}
and hence, by \eqref{eq:|x-y|2<delta(y)},
\begin{align}\label{eq:1-z_n<}
(1-z_n)=1-\sqrt{1-|\tilde z|^2}\leq |\tilde z|^2<\frac{81}{64}|x-y|^2<\frac18\delta(y).
\end{align}
Additionally, since $|z-z'|<|z-y|<\frac18|x-y|$ and by   \eqref{eq:tildeyvs|x-y|} 
we get
\begin{align}
\label{eq:z'_{n-1}>}
z'_{n-1}&>y_{n-1}-|y-z|-|z-z'|>\frac14|x-y|.
\end{align}
Then, since every point $z'$ is of the form $\(0,...,0,z'_{n-1},\frac{y_n-x_n}{y_{n-1}}z'_{n-1}+x_n \)$, we conclude from  \eqref{eq:z'_{n-1}>}, \eqref{eq:tildeyvs|x-y|} and the inequality $y_n\leq x_n$ what follows
\begin{align*}
1-z'_n&= 1-\(\frac{y_n-x_n}{y_{n-1}}z'_{n-1}+x_n\)\geq 1-\(\frac{y_n-x_n}{y_{n-1}}\frac14|x-y|+x_n\)\\
& = (1-y_n)\frac{|x-y|}{4 y_{n-1}}+(1-x_n)\(1-\frac{|x-y|}{4 y_{n-1}}\)\\
&\geq \frac14(1-y_n)\geq \frac14 \delta(y),
\end{align*}
which, combined with \eqref{eq:1-z_n<}, leads to
\begin{align*}
|z_n-z_n'|=(1-z_n')-(1-z_n)\geq \frac18\delta(y).
\end{align*}
Thus, the inequality $|z-z'|\geq \frac12|(z_1,...,z_{n-2})|+\frac12|z_n-z_n'|$ implies \eqref{eq:z-z'}. Hence we get
\begin{align*}
\frac{|z-z'|^2|x-y|^2}{|x-z||y-z|}\gtrsim\frac{|z-z'|^2|x-y|}{|y-z|}\gtrsim \frac{\big(|(z_1,...,z_{n-2})|+(\delta(y))\big)^2|x-y|}{|y-z|}.
\end{align*}
Applying this, the bound  $|\tilde z|^2\lesssim \delta_{H_x}(y)-\delta(y)$, which follows   from \eqref{eq:1-z_n<} and \eqref{eq:aux11}, and the estimates 
\begin{align*}
|x-z|\approx |x-y|, \hspace{10mm}
|y-z|\approx  \delta(y)+\left|\tilde z-\frac{\tilde y}{|y|}\right|,\hspace{20mm} z\in A_2,
\end{align*}
to \eqref{eq:integrand12}, we arrive at

\begin{align*}
I_2&\lesssim \frac{\delta(x)\delta(y)}tk(t,x,y)(\delta_{H_x}(y)-\delta(y))\\
&\ \ \ \times\int_{A_2}e^{-\frac{c}{t}\frac{\big(|(z_1,...,z_{n-2})|+(\delta(y))\big)^2|x-y|}{\left|\tilde z-\frac{\tilde y}{|y|}\right|+\delta(y)}}\Bigg(\frac1{\(\left|\tilde z-\frac{\tilde y}{|y|}\right|+\delta(y)\)^{n}} +\frac{\(\frac{|x-y|}t\)^{(n-1)/2}}{\(\left|\tilde z-\frac{\tilde y}{|y|}\right|+\delta(y)\)^{(n+1)/2}}\Bigg)dz.
\end{align*}
Similarly as in the case of $I_1$, we parametrize $A_2$ by
 $$ S_2\ni (v,w)\longrightarrow z=\(\frac{\tilde y}{|y|}\)+(v,w,\sqrt{1-\left|(v,w)+\frac{\tilde y}{|y|}\right|^2})\in A_2,$$
 for some $S_2\subset B_{n-2}(0,|x-y|/2)\times (-|x-y|/2,|x-y|/2)$. Since 
 $$|v|+|w|\geq |(v,w)|\geq \frac12(|v|+|w|),$$
  we get
 \begin{align}\label{eq:aux9}
I_2&\lesssim \frac{\delta(x)\delta(y)}tk(t,x,y)(\delta_{H_x}(y)-\delta(y))\int_{B_{n-2}(0,|x-y|/2)}\\\nonumber
&\ \ \ \ \ \ \times\int_{-|x-y|/2}^{|x-y|/2}\frac1{(|v|+|w|+\delta(y))^{n}} +\(\frac{|x-y|}t\)^{(n-1)/2}\frac{e^{-\frac ct\frac{\big(|v|+(\delta(y))\big)^2|x-y|}{|v|+\delta(y)+|w|}}}{(|v|+|w|+\delta(y))^{(n+1)/2}}dwdv.
\end{align}
We estimate the integral of the first summand as follows
\begin{align}\nonumber
&\int_{B_{n-2}(0,|x-y|/2)}e^{-{|v||x-y|}{/32t}}\int_{-|x-y|/2}^{|x-y|/2}\frac1{(|v|+|w|+\delta(y))^{n}}dwdv\\[8pt]
&\hspace{2cm}\lesssim \int_{B_{n-1}(0,|x-y|)}\frac{dz}{\(|z|+\delta(y)\)^{n}}\approx \frac1{\delta(y)}.
\label{eq:aux8}
\end{align}
Next, using \eqref{eq:intdw} with $|v|+\delta(y)$ and $n+4$ instead of $|v|$ and $n$ (note that assumptions are satisfied as $|v|+\delta(y)<\frac58|x-y|$ ), respectively, we obtain 
\begin{align*}
\int_{-|x-y|/2}^{|x-y|/2}\frac{e^{-\frac ct\frac{\big(|v|+(\delta(y))\big)^2|x-y|}{|v|+\delta(y)+|w|}}}{(|v|+|w|+\delta(y))^{(n+1)/2}}dw&\lesssim \frac{e^{-c\big(|v|+\delta(y)\big)^2/4t}}{|x-y|^{(n-1)/2}\(1+\frac{(|v|+\delta(y))^2}{t}\)}\\
&\leq \frac{e^{-c|v|^2/4t}e^{-c(\delta(y))^2/4t}}{|x-y|^{(n-1)/2}\(1+\frac{|v|^2}{t}\)},
\end{align*}
and consequently
\begin{align*}
&\(\frac{|x-y|}t\)^{(n-1)/2}\int_{B_{n-2}(0,|x-y|/2)}\int_{-|x-y|/2}^{|x-y|/2}\frac{e^{-\frac ct\frac{\big(|v|+(\delta(y))\big)^2|x-y|}{|v|+\delta(y)+|w|}}}{(|v|+|w|+\delta(y))^{(n+1)/2}}dwdv\\[8pt]
&\lesssim \frac{e^{-c(\delta(y))^2/4t}}{t^{(n-1)/2}}\int_{\R^{n-2}}\frac{e^{-c|v|^2/4t}}{1+\frac{|v|^2}{t}}dv=\frac{1}{\delta(y)}\(\frac{\delta(y)}{t}e^{-c(\delta(y))^2/4t}\)\int_{\R^{n-2}}\frac{e^{-c|v|^2/4}}{1+{|v|^2}}dv\\
&\lesssim \frac{1}{\delta(y)},
\end{align*}
which is the same bound as in \eqref{eq:aux8}. Applying this to \eqref{eq:aux9}, we obtain
\begin{align}\label{eq:I2}
I_2\lesssim \frac{\delta(x)\delta(y)}tk(t,x,y)\frac{\delta_{H_x}(y)-\delta(y)}{\delta(y)}.
\end{align}
Now,  combining \eqref{eq:r_1}, \eqref{eq:I3}, \eqref{eq:I1} and \eqref{eq:I2} gives us
 \begin{align*}
&|k_B(t,x,y)-k_{H_x}|\lesssim \frac{\delta(x)\delta(y)}{t}k(t,x,y)\(\sqrt t+|x-y|+\frac{|x-y|^{2}}{\sqrt t}+\frac{\delta_{H_x}(y)-\delta(y)}{\delta(y)}\),
\end{align*}
and the inequalities
$$\sqrt t+|x-y|+\frac{|x-y|^{2}}{\sqrt t}\leq \(t^{1/4}+\frac{|x-y|}{t^{1/4}}\)^2\leq 2\(\sqrt t+\frac{|x-y|^{2}}{\sqrt t}\)$$ 
end the proof.
\end{proof}

The next lemma shows that, under assumptions of Theorem \ref{thm:Thm2}, the heat kernel $k_B(t,x,y)$ may be approximated by $k_{H_{xy}}(t,x,y)$.

\begin{lemma}\label{lem:2.2}
There are  constants $c_1,c_2>0$ such that for $t<c_1$ and $\frac{\delta\(\frac{x+y}2\)}{\sqrt t}<c_2$ we have
$$\left|k_B(t,x,y)-k_{ H_{xy}}(t,x,y)\right|\lesssim \(\sqrt t+\sqrt{\frac{\delta\(\frac{x+y}{2}\)}{t^{1/2}}}\)k_{B}(t,x,y),\ \ \ \ x,y\in B.$$
\end{lemma}
\begin{proof}
First, let us observe that for small values of $t$ and $\frac{\delta\(\frac{x+y}2\)}{\sqrt t}$ the angle $\angle(x,y)$ between vectors $x$ and $y$ is small as well and therefore $\delta_{H_{xy}}(x)\approx \delta(x)$ and $\delta_{H_{xy}}(y)\approx \delta(y)$. Furthermore, under this assumption  formulas \eqref{eq:parallel}, \eqref{eq:MS} and \eqref{eq:kHform} imply
\begin{align}\label{eq:kapproxk}
k_B(t,x,y)\approx k_{H_{xy}}(t,x,y)\approx \frac{\delta(x)\delta(y)}{t}k(t,x,y).
\end{align}
Let us introduce sets $C_1, C_2$, depending on $x$ and $y$, as follows:
\begin{align*}
C_1&=\left\{z\in B: dist(z,P_{xy})\leq d\right\},\\
C_2&=\left\{z\in B: \left|z-\frac{x+y}{2}\right|\leq R\right\},
\end{align*}
where $P_{xy}=\partial H_{xy}$  and
\begin{align*}
d&=d(x,y):=\sqrt{\(\delta\(\frac{x+y}{2}\)+t\)\sqrt t}=t^{1/4}\sqrt{\delta\(\frac{x+y}{2}\)+t},\\
R&=R(x,y):=\sqrt{d\sqrt t}=t^{3/8}\sqrt[4]{\delta\(\frac{x+y}{2}\)+t}.
\end{align*}
Such a choice of $d$ and $R$ ensures that
\begin{align*} \frac{\delta\(\frac{x+y}{2}\)}d,\ \ \frac{d}{\sqrt t},\ \   \frac{\sqrt t}R,\ \  \frac{R}{t^{1/4}},\ \ \frac{d}{R}, \ \  \frac{R^2}{d}\rightarrow0,\ \ \ \ \ \ \ \ \ \ \text{as}\ \ t, \frac{\delta\(\frac{x+y}2\)}{\sqrt t}.
\end{align*}
In particular, it follows from the first limit and the inequality \eqref{eq:rho} that for $\frac{\delta\(\frac{x+y}2\)}{\sqrt t}$ small enough it holds $B\bs C_1\subset H_{xy}$ and $x,y\in C_1$. Then, by Chapman-Kolmogorov equation we have
\begin{align*}
&k_B(t,x,y)\\
&=\int_Bk_B(t/2,x,z)k_B(t/2,z,y)dz\\
&= \int_{H_{xy}}k_{H_{xy}}(t/2,x,z)k_{H_{xy}}(t/2,z,y)dz\\
&\ \ \ +\int_{C_2\bs C_1}\Big(k_B(t/2,x,z)k_B(t/2,z,y)-k_{H_{xy}}(t/2,x,z)k_{H_{xy}}(t/2,z,y)\Big)dz\\
&\ \ \ +\int_{C_1}k_B(t/2,x,z)k_B(t/2,z,y)dz-\int_{C_1\cap H_{xy}}k_{H_{xy}}(t/2,x,z)k_{H_{xy}}(t/2,z,y)dz\\
&\ \ \ +\int_{B\bs (C_1\cup C_2)}k_B(t/2,x,z)k_B(t/2,z,y)dz-\int_{H_{x,y}\bs(C_1\cup C_2)}k_{H_{xy}}(t/2,x,z)k_{H_{xy}}(t/2,z,y)dz\\
&:=k_{H_{xy}}(t,x,y)+I_1+I_2-I_3+I_4-I_5.
\end{align*}
We start with estimating $I_2$ and $I_3$. Directly from the definition of the set $C_1$ and the formula \eqref{eq:kHest} we get 
$$k_{H_{xy}}(t/2,x,z)\lesssim \frac{\delta(x)d}tk(t/2,x,z),\ \ \ \ z\in C_1\cap H_{xy}, \ x\in B.$$
Next, for $z\in C_1$ we have $\delta(z)<2d$. Additionally, since $C_1$ is a segment of a ball cut by a hyperplane, it holds $|x-z|, |y-z|<diam(C_1)\leq 2\sqrt {2d}$, and  hence by \eqref{eq:MS} we obtain for $\delta\(\frac{x+y}2\)/\sqrt t$ small enough
$$k_B(t/2,x,z)\lesssim \frac{\delta(x)d}{t}\(1+\frac{d^2}{t}\)k(t/2,x,z)\approx \frac{\delta(x)d}{t}k(t/2,x,z),\ \ \ \ z,x\in C_1.$$
Clearly, the same inequalities hold for $x$ replaced by $y$. Thus
\begin{align}\nonumber
I_2,I_3&\lesssim \frac{\delta(x)\delta(y)}{t^2}d^2\int_{C_1}k(t/2,x,z)k(t/2,z,y)dz\\\label{eq:finalI23}
&\leq \frac{d^2}{t}\frac{\delta(x)\delta(y)}{t}k(t,x,y)\approx  \(\frac{\delta\(\frac{x+y}{2}\)}{\sqrt t}+\sqrt t\)k_{H_{xy}}(t,x,y).
\end{align}
Furthermore, from \eqref{eq:kHest} and \eqref{eq:MS} we get
$$k_B(t/2,x,z), k_{H_{xy}}(t/2,x,z)\lesssim \frac{\delta(x)}{t^2}k(t/2,x,z),\ \ \ x,z\in B\cap H_{x,y},\ t<1.$$
Using this and Corollary \ref{cor:CKRR}, we obtain
\begin{align}\nonumber
I_4,I_5&\lesssim \frac{\delta(x)\delta(y)}{t^4}\int_{ C_2^c}k(t/2,x,z)k(t/2,z,y)dz\\\nonumber
&\lesssim \frac{e^{-R^2/2t}}{t^3}\frac{\delta(x)\delta(y)}{t}k(t,x,y)\lesssim \frac{e^{-t^{1/4}/4}}{t^3}e^{-t^{1/4}/4}k_{H_{xy}}(t,x,y)\\\label{eq:finalI45}
&\lesssim {e^{-t^{1/4}/4}}k_{H_{xy}}(t,x,y).
\end{align}
We pass to the crucial integral $I_1$. As a first step we will show  that $k_B(t/2,x,z)$ and $k_B(t/2,z,y)$ may be somehow switched to $k_{H_x}(t/2,x,z)$ and $k_{H_y}(t/2,z,y)$, respectively. For this purpose we will employ Lemma \ref{lem:2.1}. Note that assumptions are satisfied since
\begin{align}\label{eq:delta(z)>d/2}
\delta(z)\geq \frac12 d\geq \delta(x),\delta(y),\ \ \ z\in C_2\bs C_1,
\end{align}
for $\delta\(\frac{x+y}2\)/\sqrt t$ small enough. The latter inequality is clear. To explain the first one, let us  denote by $P_{w_0}$ the hyperplane parallel to $P_{xy}$ tangent to $B$ at $w_0\notin H_{xy}$.  By simple calculation we have
\begin{align*}
\delta(z)&=1-\sqrt{(1-\delta_{H_{w_0}}(z))^2+\(|z-w_0|^2-\(\delta_{H_{w_0}}(z)\)^2\)}\\
&=\frac{2\delta_{H_{w_0}}-|z-w_0|^2}{1+\sqrt{(1-\delta_{H_{w_0}}(z))^2+\(|z-w_0|^2-\(\delta_{H_{w_0}}(z)\)^2\)}}\\
&\geq \delta_{H_{w_0}}(z)-\frac12|z-w_0|^2.
\end{align*} 
Since $z\in B\bs C_1$ then $\delta_{w_0}(z)\geq dist(z,P_{xy})\geq d$, so we need to show that $|z-w_0|^2\leq d/2$ for $z\in C_2\bs C_1$. Indeed, by \eqref{eq:parallel} and \eqref{eq:x_0-y_0}, we have  for $z\in C_2$ 
\begin{align*}
|z-w_0|&\leq \left|z-\frac{x+y}{2}\right|+\left|\frac{x+y}{2}-x\right|+\left|x-\frac{x}{|x|}\right|+\left|\frac{x}{|x|}-w_0\right|\\
&\leq  R+\frac12|x-y|+\delta(x)+\left|\frac x{|x|}-\frac{y}{|y|}\right|\leq R+2(\delta(x)+|x-y|+\delta(y))\\
&\leq R+4\sqrt 6\sqrt{\delta\(\frac{x+y}2\)}\leq \sqrt d \(t^{1/4}+4\sqrt 6\sqrt[4]{\frac{\delta\(\frac{x+y}2\)}{\sqrt t}}\),
\end{align*}
where the expression in brackets is smaller than $1$ if $t$ and $\frac{\delta\(\frac{x+y}2\)}{\sqrt t}$ are sufficiently small, as required.
 Furthermore, by simple geometry we have $\delta_{H_x}(z)-\delta_B(z)\leq \delta_{H_x}\(\frac{z}{|z|}\)$ and  $\delta_{H_x}\(\frac{z}{|z|}\)=\frac12\left|\frac{x}{|x|}-\frac{z}{|z|}\right|^2 $. Thus, for $z\in C_2$
\begin{align}\nonumber
\delta_{H_x}(z)-\delta_B(z)&\leq \frac12\left|\frac{x}{|x|}-\frac{z}{|z|}\right|^2\leq \frac12\(\left|\frac{x}{|x|}-z\right|+\delta(z)\)^2\leq 2 \left|\frac{x}{|x|}-z\right|^2\\\nonumber
&\leq 2\(\delta(x)+\left|x-\frac{x+y}{2}\right|+\left|\frac{x+y}{2}-z\right|\)^2\leq 4\(\delta(x)+\frac12|x-y|+R\)^2\\\label{eq:aux13}
&\leq 4\(\frac12|x-y|+2R\)^2,
\end{align} 
where we also used $\delta(x)\leq \delta\(\frac{x+y}2\)\leq R$. Hence, by \eqref{eq:delta(z)>d/2},
\begin{align*}
\frac{\delta_{H_x}(z)-\delta_B(z)}{\delta(z)}&\lesssim \frac{\delta\(\frac{x+y}2\)+R^2}{d}\leq \sqrt{\frac{\delta\(\frac{x+y}2\)}{\sqrt t}}+\sqrt t.
\end{align*}
Similarly, 
\begin{align}\label{eq:|x-z|^2}
\frac{|x-z|^2}{\sqrt t}\leq \frac{\(\frac12|x-y|+R\)^2}{\sqrt t}\lesssim \frac{|x-y|^2}{\sqrt t}+d\lesssim \frac{\delta\(\frac{x+y}2\)}{\sqrt t}+\sqrt t.
\end{align}
Thus, applying Lemma \ref{lem:2.1}, we get
\begin{align*}
|k_B(t/2,x,z)-k_{H_x}(t/2,x,z)|&\lesssim  \frac{\delta(x)\delta(z)}{t}k(t/2,x,z)\(\sqrt t+\sqrt {\frac{\delta\(\frac{x+y}2\)}{\sqrt t}}\),
\end{align*}
where $z\in C_2\bs C_1$ and $\delta(x)/\sqrt t$ is small enough. Furthermore, by \eqref{eq:MS} and \eqref{eq:|x-z|^2},
$$k_B(t/2,x,z)\lesssim \frac{\delta(x)\delta(z)}{t}.$$
Thus, for $z\in C_2\bs C_1$ we have
\begin{align*}
&\left|k_B(t/2,x,z)k_B(t/2,z,y)-k_{H_{x}}(t/2,x,z)k_{H_{y}}(t/2,z,y)\right|\\
&=\left|k_B(t/2,x,z)\(k_B(t/2,z,y)-k_{H_y}(t/2,z,y)\)\right.\\
&\ \ \ \ \ \left.+\(k_B(t/2,x,z)-k_{H_y}(t/2,x,z)\)\([k_{H_y}(t/2,z,y)-k_{B}(t/2,z,y)]+k_{B}(t/2,z,y)\)\right|\\
&\lesssim\(\sqrt t+\sqrt {\frac{\delta\(\frac{x+y}2\)}{\sqrt t}}\)\frac{\delta(x)\delta(y)}{t^2}\(\delta(z)\)^2k(t/2,x,z)k(t/2,z,y).
\end{align*}
By \eqref{eq:delta(z)>d/2} and \eqref{eq:aux12} we   bound $\delta(z)\lesssim \delta_{H_{xy}}(z)$ and consequently, by Proposition \ref{prop:CKHB}, 
\begin{align}\nonumber
&\left|\int_{C_2\bs C_1}k_B(t/2,x,z)k_B(t/2,z,y)-k_{H_{x}}(t/2,x,z)k_{H_{y}}(t/2,z,y)\right|\\\nonumber
&\lesssim \(\sqrt t+\sqrt {\frac{\delta\(\frac{x+y}2\)}{\sqrt t}}\)\frac{\delta(x)\delta(y)}{t^2}\int_{H_{xy}}\(\delta_{H_{xy}}(z)\)^2k(t/2,x,z)k(t/2,z,y)dz\\\label{eq:aux10}
&\approx \(\sqrt t+\sqrt {\frac{\delta\(\frac{x+y}2\)}{\sqrt t}}\)k_{H_{xy}}(t,x,y).
\end{align}
The next step is to stimate the difference between  expressions $k_{H_{x}}(t/2,x,z)k_{H_{y}}(t/2,z,y)$ and $k_{H_{xy}}(t/2,x,z)k_{H_{xy}}(t/2,z,y)$. By \eqref{eq:kHform}, we may write
\begin{align}\nonumber
|k_{H_{xy}}(t/2,x,z)-k_{H_x}(t/2,x,z)|&=k(t/2,x,z)e^{-2\delta(x)\delta_{H_x}(z)/t}\left|1-e^{-2\delta(x)\(\delta_{H_{xy}}(z)-\delta_{H_{x}}(z)\)/t}\right|\\
\label{eq:kHxx-kHx}
&\leq k(t/2,x,z)\left|1-e^{-2\delta(x)\(\delta_{H_{xy}}(z)-\delta_{H_{x}}(z)\)/t}\right|.
\end{align}
Let us denote by $l_z$ the line perpendicular to $H_{xy}$  and containing $z$, and by $w_x=l_z\cap P_x$, $w_{xy}=l_z\cap P_{xy}$ points that are intersections of $l_z$ with $P_x$ and $P_{xy}$, respectively. Then we have 
\begin{align*}
\delta_{H_{xy}}(z)=|z-w_{xy}|, \hspace{20mm}\delta_{H_x}(z),
=\cos\angle(P_{xy},P_x)|z-w_x|
\end{align*}   
and consequently
\begin{align*}
\left|\delta_{H_x}(z)-\delta_{H_{xy}}(z)\right|&=
\left|\(\cos\angle(P_{xy},P_x)-1\)\delta_{H_{xy}}(z)+\cos\angle(P_{xy},P_x)\(|z-w_z|-\delta_{H_{xy}}(z)\)\right|\\
&\leq \tan^2\angle(P_{xy},P_x)+\left||z-w_x|-\delta_{H_{xy}}(z)\right|.
\end{align*}  
 The inequality \eqref{eq:x_0-y_0} gives us
\begin{align}\label{eq:tan}
\tan\angle(P_{xy},P_x)=\left.\frac{d}{dv}\sqrt{1-v^2}\right|_{v=-\frac12\left|\frac{x}{|x|}-\frac{y}{|y|}\right|}\leq \left|\frac{x}{|x|}-\frac{y}{|y|}\right|\leq2\sqrt6\,\sqrt{\delta\(\frac{x+y}{2}\)}.
\end{align} 
Furthermore, since $x/|x|\in P_{x}\cap P_{xy}$ and by estimate of $\left|\frac x{|x|}-z\right|$ in \eqref{eq:aux13}, we get
\begin{align*}
\left||z-w_x|-\delta_{H_{xy}}(z)\right|&=\big||z-w_x|-|z-w_{xy}|\big|\leq |w_x-w_{xy}|\\
&\leq\tan\angle(P_{xy},P_x)\left|\frac{x}{|x|}-w_{xy}\right|\leq\tan\angle(P_{xy},P_x)\left|\frac{x}{|x|}-z\right|\\
&\lesssim\sqrt{\delta\(\frac{x+y}{2}\)}\(\sqrt{\delta\(\frac{x+y}{2}\)}+R\).
\end{align*} 
and therefore,  for $\delta\(\frac{x+y}2\)/\sqrt t$ small enough we have
\begin{align*}
\frac{\delta(x)}{t}\left|\delta_{H_{xy}}(z)-\delta_{H_x}(z)\right|&\lesssim   \frac{\delta(x)}{t}\(\delta\(\frac{x+y}{2}\)+\sqrt{\delta\(\frac{x+y}{2}\)}\(\sqrt{\delta\(\frac{x+y}{2}\)}+R\)\)\\
&\lesssim\frac{\delta(x)}{t}\(\delta\(\frac{x+y}{2}\)+R^2\)\\
&=\frac{\delta(x)}{t^{1/2}}\(\frac{\delta\(\frac{x+y}{2}\)}{t^{1/2}}+t^{1/4}\sqrt{\frac{\delta\(\frac{x+y}{2}\)}{t^{1/2}}}+\sqrt t\)\\
&=\frac{\delta(x)}{t^{1/2}}\(\frac{\delta\(\frac{x+y}{2}\)}{t^{1/2}}+\sqrt t\)<1,
\end{align*}
where we used  $a+ab\leq 2a^2+b^2$, $a,b\geq 0$ in the second inequality.
Applying this to \eqref{eq:kHxx-kHx} and using $\delta(x)\approx \delta_{H_{xy}(x)}$ and  the estimate $|1-e^v|\lesssim v$, $|v|<1$, we obtain
\begin{align*}
|k_{H_{xy}}(t/2,x,z)-k_{H_x}(t/2,x,z)|\lesssim k(t/2,x,z)\frac{\delta_{H_{xy}}(x)}{t^{1/2}}\(\frac{\delta\(\frac{x+y}{2}\)}{t^{1/2}}+\sqrt t\).
\end{align*}
Hence,  using this and \eqref{eq:kHest}, we may write for $\delta\(\frac{x+y}{2}\)/\sqrt t$ small enough
\begin{align*}
&\left|k_{H_x}(t/2,x,z)k_{H_y}(t/2,z,y)-k_{H_{xy}}(t/2,x,z)k_{H_{xy}}(t/2,z,y)\right|\\
&\leq\left|(k_{H_{x}}(t/2,x,z)-k_{H_{xy}}(t/2,x,z))\(k_{H_y}(t/2,z,y)-k_{H_{xy}}(t/2,z,y)\)\right|\\
&\ \ \ +\left|k_{H_{xy}}(t/2,x,z)\(k_{H_y}(t/2,z,y)-k_{H_{xy}}(t/2,z,y)\)\right|\\
&\ \ \ +\left|\(k_{H_x}(t/2,x,z)-k_{H_{xy}}(t/2,x,z)\)k_{H_{xy}}(t/2,z,y)\right|\\
&\lesssim \frac{\delta_{H_{xy}}(x)\delta_{H_{xy}}(y)}{t}k(t/2,x,z)k(t/2,z,y)\(\frac{\delta\(\frac{x+y}{2}\)}{t^{1/2}}+\sqrt t\)\(1+\frac{\delta_{H_{xy}}(z)}{t^{1/2}}
\).
\end{align*}
Thus, applying Proposition \ref{prop:CKHB} with $r=0$ and $\beta=0,\frac12$ as well as the estimate \eqref{eq:kapproxk}, we conclude
\begin{align*}
&\int_{C_2\bs C_1}\left|k_{H_x}(t/2,x,z)k_{H_y}(t/2,z,y)-k_{H_{xy}}(t/2,x,z)k_{H_{xy}}(t/2,z,y)\right|dz\\
&\lesssim \(\frac{\delta\(\frac{x+y}{2}\)}{t^{1/2}}+\sqrt t\)k_{H_{xy}}(t,x,y).
\end{align*}
Combining this with \eqref{eq:aux10} we get
\begin{align}\label{eq:finalI1}I_1\lesssim \(\sqrt{\frac{\delta\(\frac{x+y}{2}\)}{t^{1/2}}}+ \sqrt t\)k_{H_{xy}}(t,x,y).
\end{align}
The assertion of the lemma follows now from \eqref{eq:finalI23}, \eqref{eq:finalI45}, \eqref{eq:finalI1} and \eqref{eq:kapproxk}.  
\end{proof}
Theorem \ref{thm:Thm2} is a direct consequence of Lemmas \ref{lem:2.2} and \ref{lem:kH-ddk}.

\begin{lemma}\label{lem:kH-ddk}
There are constants $c_1, c_2>0$ such that for $t<c_1$ and $\frac{\delta\(\frac{x+y}2\)}{\sqrt t}<c_1$ we have
\begin{align*}
\left|k_{H_{xy}}(t,x,y)-\frac{\delta(x)\delta(y)}tk(t,x,y)\right|\lesssim  \frac{\delta\(\frac{x+y}{2}\)}{\sqrt t}\(\sqrt t+  \(\frac{\delta\(\frac{x+y}{2}\)}{\sqrt t}\)^2\)k_B(t,x,y).
\end{align*}
\end{lemma}
\begin{proof}
Due to the inequalities $\delta_{H_{xy}}(x)\leq \delta(x)$, $\delta_{H_{xy}}(y)\leq \delta(y)$ and $1-e^{-u}\leq u, u\geq0$, we get
\begin{align*}
&\left|\(1-e^{-\delta_{H_{xy}}(x)\delta_{H_{xy}}(y)/t}\)-\(1-e^{-\delta(x)\delta(y)/t}\)\right|\\
&=e^{-\delta_{H_{xy}}(x)\delta_{H_{xy}}(y)/t}\(1-e^{-\big(\delta(x)\delta(y)-\delta_{H_{xy}}(x)\delta_{H_{xy}}(y)\big)/t}\)\\
&\leq \frac1t \big(\delta(x)\delta(y)-\delta_{H_{xy}}(x)\delta_{H_{xy}}(y)\big)\\
&= \frac1t \bigg(\big(\delta(x)-\delta_{H_{xy}}(x)\big)\delta(y)+\delta_{H_{xy}}(x)\big(\delta(y)-\delta_{H_{xy}}(y)\big)\bigg)\\
&\leq \frac1t \bigg(\big(\delta(x)-\delta_{H_{xy}}(x)\big)\delta(y)+\delta(x)\big(\delta(y)-\delta_{H_{xy}}(y)\big)\bigg).
\end{align*}
Since $\delta_{H_{xy}}(x)=\cos\angle(P_{xy},P_x)\delta(x)$ and $\delta_{H_{xy}}(y)=\cos\angle(P_{xy},P_x)\delta(y)$ and by \eqref{eq:tan}, for $\delta\(\frac{x+y}{2}\)/\sqrt t$ small enough we obtain
\begin{align*}
\left|\(1-e^{-\delta_{H_{xy}}(x)\delta_{H_{xy}}(y)/t}\)-\(1-e^{-\delta(x)\delta(y)/t}\)\right|&\leq \(1-\cos\angle(P_{xy},P_x)\)\frac{2\delta(x)\delta(y)}t \\
&\leq\tan^2\angle(P_{xy},P_x) \frac{2\delta(x)\delta(y)}t \\
&\lesssim  \delta\(\frac{x+y}{2}\)\frac{\delta(x)\delta(y)}t,
\end{align*}
and consequently
\begin{align*}
&\left|\(1-e^{-\delta_{H_{xy}}(x)\delta_{H_{xy}}(y)/t}\)-\frac{\delta(x)\delta(y)}{t}\right|\\
&=\left|\(1-e^{-\delta_{H_{xy}}(x)\delta_{H_{xy}}(y)/t}\)-\(1-e^{-\delta(x)\delta(y)/t}\)\right|+\left|\(1-e^{-\delta(x)\delta(y)/t}\)-\frac{\delta(x)\delta(y)}{t}\right|\\
&\lesssim  \frac{\delta(x)\delta(y)}t\(\delta\(\frac{x+y}{2}\)+  \frac{\delta(x)\delta(y)}t\)\lesssim  \frac{\delta(x)\delta(y)}t\(\sqrt t+  \frac{\(\delta\(\frac{x+y}{2}\)\)^2}t\).
\end{align*}
Hence, in view of  \eqref{eq:kHform} and \eqref{eq:kapproxk}, the proof is complete.
\end{proof}
\section{Appendix}
The first proposition and its corollaries deal with Chapman-Kolmogorow identity-related integrals.

\begin{proposition}\label{prop:CKHB}
Let $\beta\geq0$ and $\alpha\in (0,1)$. Then there is a constant $ c_{n,\beta}$ depending only on $n$ and $\beta $ such that  for any half-space $H\subset \R^n$ and  $r\geq 0$ it holds   
\begin{align*}
&\int_{H\bs B((1-\alpha)x+\alpha y ,r)} \kernel(\alpha t,x,z)\kernel((1-\alpha)t,z,y)\(\delta_H(z)\)^\beta dz\\
&\leq c_{n,\beta} \,k(t,x,y)t^{\beta/2}\exp\left(-\frac{r^2}{8\alpha(1-\alpha)t}\right)\(1+\frac{\delta_H(x)+\delta_H(y)}{\sqrt t}\)^\beta.
\end{align*}
\end{proposition}

\begin{proof}
Due to translational and rotational invariance of the heat kernel $k(t,x,y)$, we may assume that $x=(-\alpha|x-y|,0,...,0)$,  $y=((1-\alpha)|x-y|,0,...,0)$. Then we have $B((1-\alpha) x+\alpha y,r)=B(0,r)$ and  
$$\delta_H(z)\leq |z|+\delta_H(x)+\delta_H(y),$$
 as well as
\begin{eqnarray*}
  \kernel(\alpha t,x,z)\kernel((1-\alpha)t,z,y)=k(t,x,y) \frac{\exp\left(-\frac{|z|^2}{4\alpha(1-\alpha)t}\right)}{(4\alpha(1-\alpha)\pi t)^{n/2}}\/.
\end{eqnarray*}
Hence we get
\begin{align*}
&\int_{H-B((1-\alpha)x+\alpha y ,r)} \kernel(\alpha t,x,z)\kernel((1-\alpha)t,z,y)\(\delta_H(z)\)^\beta dz\\
&\leq k(t,x,y)\int_{B(0,r)^c} \frac{\exp\left(-\frac{|z|^2}{4\alpha(1-\alpha)t}\right)}{(4\alpha(1-\alpha)\pi t)^{n/2}}\(|z|+\delta_H(x)+\delta_H(y)\)^\beta dz\\
 &=\frac{n}{\Gamma\(n/2\)} k(t,x,y)\int_r^\infty \frac{\exp\left(-\frac{u^2}{4\alpha(1-\alpha)t}\right)}{ (4\alpha(1-\alpha)t)^{n/2}}\(u+\delta_H(x)+\delta_H(y)\)^\beta u^{n-1}du\\
  &=\frac{n}{\Gamma\(n/2\)} k(t,x,y)t^{\beta/2}\int_{r/\sqrt{4\alpha(1-\alpha) t}}^\infty e^{-v^2}\(\sqrt{4\alpha(1-\alpha)}v+\frac{\delta_H(x)+\delta_H(y)}{\sqrt t}\)^\beta v^{n-1}dv\\
    &\leq \frac{n}{\Gamma\(n/2\)}2^{\beta} k(t,x,y)t^{\beta/2}\exp\left(-\frac{r^2}{8\alpha(1-\alpha)t}\right)\int_{0}^\infty e^{-v^2/2}\(v^{\beta}+\(\frac{\delta_H(x)+\delta_H(y)}{\sqrt t}\)^\beta\) v^{n-1}dv\\
        &\leq c_{n,\beta}  k(t,x,y)t^{\beta/2}\exp\left(-\frac{r^2}{8\alpha(1-\alpha)t}\right)\(1+\frac{\delta_H(x)+\delta_H(y)}{\sqrt t}\)^\beta,
\end{align*}
for some $c_{n,\beta}>0$.

\end{proof}

\begin{corollary}\label{cor:CKHH}
Let $\alpha\in (0,1)$ and $r>0$. Then for any half-space $H\subseteq \R^n$  
\begin{align*}
&\int_{H\bs B((1-\alpha)x+\alpha y ,r)} \kernel_H(\alpha t,x,z)\kernel_H((1-\alpha)t,z,y)dz\lesssim \,k_H(t,x,y)\frac{\exp\left(-\frac{r^2}{8\alpha(1-\alpha)t}\right)}{\alpha(1-\alpha)}.
\end{align*}
\end{corollary}
\begin{proof}
Without loss of generality we may assume $\delta_H(x)\geq \delta_H(y)$. If $\delta_H(x)/\sqrt t \leq 1$, then \eqref{eq:kHest} and Proposition \ref{prop:CKHB}  give us
\begin{align*}
&\int_{H\bs B((1-\alpha)x+\alpha y ,r)} \kernel_H(\alpha t,x,z)\kernel_H((1-\alpha)t,z,y)dz\\
&\leq \frac{\delta_H(x)\delta_H(y)}{\alpha(1-\alpha)t^2}\int_{H\bs B((1-\alpha)x+\alpha y ,r)} \kernel(\alpha t,x,z)\kernel((1-\alpha)t,z,y)\(\delta_H(z)\)^2dz\\
&\lesssim \frac{\delta_H(x)\delta_H(y)}{\alpha(1-\alpha)t^2}t\(1+\frac{\delta_H(x)}{\sqrt t}\)^2k(t,x,y)\exp\left(-\frac{r^2}{8\alpha(1-\alpha)t}\right)\\
&\approx \frac{\delta_H(x)\delta_H(y)}{t}k(t,x,y)\frac{\exp\left(-\frac{r^2}{8\alpha(1-\alpha)t}\right)}{\alpha(1-\alpha)}.
\end{align*}
In case when $\delta_H(x)/\sqrt \geq 1$ we proceed similarly, but estimate $k_H(t,x,z)$ just by $k(t,x,y)$, and obtain the same bound:
\begin{align*}
&\int_{H\bs B((1-\alpha)x+\alpha y ,r)} \kernel_H(\alpha t,x,z)\kernel_H((1-\alpha)t,z,y)dz\\
&\leq \frac{\delta_H(y)}{(1-\alpha)t}\int_{H\bs B((1-\alpha)x+\alpha y ,r)} \kernel(\alpha t,x,z)\kernel((1-\alpha)t,z,y)\delta_H(z)dz\\
&\lesssim \frac{\delta_H(y)}{(1-\alpha)t}\sqrt t\(1+\frac{\delta_H(x)}{\sqrt t}\)k(t,x,y)\exp\left(-\frac{r^2}{8\alpha(1-\alpha)t}\right)\\
&\approx \frac{\delta_H(x)\delta_H(y)}{t}k(t,x,y)\frac{\exp\left(-\frac{r^2}{8\alpha(1-\alpha)t}\right)}{\alpha(1-\alpha)}.
\end{align*}
Finally, estimating  $k_H(t,x,z)$ and $k_H(t,z,y)$ by $k(t,x,z)$ and $k(t,z,y)$, respectively, we get
\begin{align*}
&\int_{H\bs B((1-\alpha)x+\alpha y ,r)} \kernel_H(\alpha t,x,z)\kernel_H((1-\alpha)t,z,y)dz\\
&\leq \int_{H\bs B((1-\alpha)x+\alpha y ,r)} \kernel(\alpha t,x,z)\kernel((1-\alpha)t,z,y)dz\\
&\lesssim k(t,x,y)\exp\left(-\frac{r^2}{8\alpha(1-\alpha)t}\right).
\end{align*}
Combining both of the  bounds we arrive at
\begin{align*}
&\int_{H\bs B((1-\alpha)x+\alpha y ,r)} \kernel_H(\alpha t,x,z)\kernel_H((1-\alpha)t,z,y)dz\\
&\lesssim \(1\wedge \frac{\delta_H(x)\delta_H(y)}{t}\)k(t,x,y)\frac{\exp\left(-\frac{r^2}{8\alpha(1-\alpha)t}\right)}{\alpha(1-\alpha)}\\
&\approx k_H(t,x,y)\frac{\exp\left(-\frac{r^2}{8\alpha(1-\alpha)t}\right)}{\alpha(1-\alpha)},
\end{align*}
which ends the proof.
\end{proof}
Considering $H=\R^n$ one can obtain a slightly better bound than the one in Corollary \ref{cor:CKHH}. Namely, since the constant $c_{n,\beta}$ in Proposition \ref{prop:CKHB} does not depend on a set, taking $\beta=0$  and approaching $\R^n$ by increasing sequence of half-spaces, we obtain
\begin{corollary}\label{cor:CKRR}
Let $\alpha\in (0,1)$ and $r>0$. Then we have
\begin{align*}
&\int_{B((1-\alpha)x+\alpha y ,r)^c} \kernel(\alpha t,x,z)\kernel((1-\alpha)t,z,y)\lesssim \,k(t,x,y)\exp\left(-\frac{r^2}{8\alpha(1-\alpha)t}\right).
\end{align*}
\end{corollary}

\vspace{5mm}
In the ext proposition we estimate the integral 
$$I_{\alpha,\beta}(t,a,b):=\int_0^t \frac1{s^\alpha(t-s)^\beta}\exp\({-\frac{a^2}{s}-\frac{b^2}{t-s}}\)ds,$$
which plays a crucial role  in the proof of Lemma \ref{lem:2.1}. This kind of integrals appear often when applying strong Markov property for Brownian motion. Additionally, it represents, up to a multiplicative constant, the density of a convolution of two inverse-gamma distribution.
\begin{proposition}\label{prop:estints}
Fix $\alpha, \beta>\frac32$. For $a,b,t>0$  we have
\begin{align*}
I_{\alpha,\beta}(t,a,b)\approx e^{-\frac{(a+b)^2}{t}}\(\frac{\(\frac t{a^2}\)^{\alpha-1}+\(\frac t{b^2}\)^{\beta-1}}{t^{\alpha+\beta-1}}+\frac{(a+b)^{\alpha+\beta-2}}{t^{\alpha+\beta-1}}\frac{\sqrt t}{a^{\alpha-1}b^{\beta-1}\sqrt{t+ab}}\),
\end{align*}
where the constants in estimates depend only on $\alpha$ and $\beta$.
\end{proposition}
\begin{proof}
Throughout this proof constants in notations $\approx$ and $\lesssim $ depend only on $\alpha$ and $\beta$. By the equality
$$-\frac{a^2}{s}-\frac{b^2}{t-s}=-\frac{(a+b)^2}{t}\(1+\frac{\(s-\frac{a}{a+b}t\)^2}{s(t-s)}\),$$
we get
\begin{align*}
I_{\alpha,\beta}(t,a,b)&=e^{-(a+b)^2/t}\int_0^t\frac1{s^\alpha(t-s)^\beta}\exp\({-\frac{(a+b)^2\(s-\frac{a}{a+b}t\)^2}{st(t-s)}}\)ds\\
&=e^{-(a+b)^2/t}\int_0^{\frac{a}{a+b}t}+\int_{\frac{a}{a+b}t}^t...ds.
\end{align*}
Substituting $s=\frac{a}{a+b}t\(1-u\)$ and $s=u\frac b{a+b}t+\frac{a}{a+b}t$, respectively, we obtain
\begin{align*}
I_{\alpha,\beta}(t,a,b)&=e^{-(a+b)^2/t}\(\frac{a+b}{at}\)^{\alpha+\beta-1}\int_0^1\frac1{(1-u)^\alpha\(u+\frac{b}{a}\)^\beta}\exp\({-\frac{(a+b)^2u^2}{(1-u)t\(u+\frac ba\)}}\)du\\
&\ \ \ +e^{-(a+b)^2/t}\(\frac{a+b}{bt}\)^{\alpha+\beta-1}\int_0^1\frac1{\(u+\frac{a}{b}\)^\alpha(1-u)^\beta}\exp\({-\frac{(a+b)^2u^2}{(1-u)t\(u+\frac ab\)}}\)du.
\end{align*}
This let us write
\begin{align}\label{eq:I=J+J}
I_{\alpha,\beta}(t,a,b)=e^{-(a+b)^2/t}\(\frac{a+b}{t}\)^{\alpha+\beta-1}\(a^{-\alpha-\beta+1}J_{\alpha,\beta}(t,a,b)+b^{-\alpha-\beta+1}J_{\beta,\alpha}(t,b,a)\),
\end{align}
where
\begin{align*}
J_{\alpha,\beta}(t,a,b)&=\int_0^1\frac1{(1-u)^\alpha\(u+\frac{b}{a}\)^\beta}\exp\({-\frac{(a+b)^2u^2}{(1-u)t\(u+\frac ba\)}}\)du\\
&=\int_0^{1/2}+\int_{1/2}^1...du:=J^{(1)}_{\alpha,\beta}(t,a,b)+J^{(2)}_{\alpha,\beta}(t,a,b).
\end{align*}
For $b/a>1/4$ we have
\begin{align*}
J^{(1)}_{\alpha,\beta}(t,a,b)&\lesssim \int_0^{1/2}\frac1{\(\frac{b}{a}\)^\beta}\exp\({-c_1\frac{(a+b)^2u^2}{t\(\frac ba\)}}\)du\\
&=\(\frac ab\)^{\beta}\frac{\sqrt{t\frac{b}{a}}}{a+b}\int_0^{(a+b)/2\sqrt{tb/a}}e^{-c_1r^2}dr\\
&\approx \(\frac ab\)^{\beta}\frac{\sqrt{t\frac{b}{a}}}{a+b}\(1\wedge \frac{a+b}{\sqrt{t\frac{b}{a}}}\)=\(\frac ab\)^{\beta}\(\frac{\sqrt{t\frac{b}{a}}}{a+b}\wedge 1\)\\
&\approx \(\frac ab\)^{\beta}\({\sqrt{\frac{t}{ab}}}\wedge 1\)\approx \(\frac ab\)^{\beta}\sqrt{\frac{t}{t+ab}}\\
&\approx \(\frac{a}{b}\)^{\beta-1}\frac{a}{a+b}\sqrt{\frac{t}{ab+t}},
\end{align*}
where $c_1=\frac13$. For $b/a\leq1/4$ we get
\begin{align*}
J^{(1)}_{\alpha,\beta}(t,a,b)
&\lesssim \int_0^{b/a}\frac1{\(\frac{b}{a}\)^\beta}\exp\({-c_2\frac{(a+b)^2u^2}{t\(\frac ba\)}}\)du+\int_{b/a}^{1/2}\frac1{u^\beta}\exp\({-c_2\frac{(a+b)^2u}{t}}\)du,\end{align*}
where $c_2=\frac12$. Substituting $u=\sqrt{tb}/(a+b)\sqrt a$, we estimate the first integral by
\begin{align*}
\(\frac ab\)^{\beta}\frac{\sqrt{t\frac{b}{a}}}{a+b}\int_0^{\frac ba (a+b)/2\sqrt{tb/a}}e^{-c_2r^2}dr\approx\(\frac ab\)^{\beta}\frac{\sqrt{t\frac{b}{a}}}{a+b}\(1\wedge \frac ba\frac{a+b}{\sqrt{t\frac{b}{a}}}\)\approx\(\frac ab\)^{\beta-1}\( \sqrt{\frac t{ab}}\wedge 1\).
\end{align*}
Furthermore,
\begin{align*}
\int_{b/a}^{1/2}\frac1{u^\beta}\exp\({-c_2\frac{(a+b)^2u}{t}}\)du&\leq \exp\({-c_2\frac{(a+b)^2\frac ba}{t}}\)\int_{b/a}^{\infty}\frac1{u^\beta}du\leq \(\frac ab\)^{\beta-1} \frac{e^{-c_2ab/t}}{\beta-1},
\end{align*}
which is dominated by the first integral, and therefore we have
\begin{align*}
J^{(1)}_{\alpha,\beta}(t,a,b)\lesssim \(\frac ab\)^{\beta-1}\( \sqrt{\frac t{ab}}\wedge 1\)\approx \(\frac{a}{b}\)^{\beta-1}\frac{a}{a+b}\sqrt{\frac{t}{ab+t}}.
\end{align*}
 On the other hand, taking $c_1=c_2=2$, we obtain opposite inequalities, and hence for any value of $\frac ba$ it holds
\begin{align}\label{eq:J^1}J^{(1)}_{\alpha,\beta}(t,a,b)\approx \(\frac{a}{b}\)^{\beta-1}\frac{a}{a+b}\sqrt{\frac{t}{ab+t}}.\end{align}
Let us now estimate the integral $J^{(2)}_{\alpha,\beta}(t,a,b)$. Since $\alpha\geq \frac32>1$, we have
\begin{align*}
J^{(2)}_{\alpha,\beta}(t,a,b)&\lesssim \int_{1/2}^1\frac1{(1-u)^\alpha\(1+\frac{b}{a}\)^\beta}\exp\({-c_3\frac{(a+b)^2}{(1-u)t\(1+\frac ba\)}}\)du\\
&= \frac{t^{\alpha-1}}{a^{\alpha-\beta-1}(a+b)^{\alpha+\beta-1}}\int_{a(a+b)/t}^{\infty}r^{\alpha-2}e^{-c_3r}dr\\
&\approx \frac{t^{\alpha-1}}{a^{\alpha-\beta-1}(a+b)^{\alpha+\beta-1}}\(1+\frac{a(a+b)}{t}\)^{\alpha-2}e^{-c_3a(a+b)/t},
\end{align*}
where $c_3=\frac14$. If we teke $c_3=2$, we get an  opposite inequality, and hence
\begin{align}\label{eq:J^2}
\frac{t^{\alpha-1}}{a^{\alpha-\beta-1}(a+b)^{\alpha+\beta-1}}&\(1+\frac{a(a+b)}{t}\)^{\alpha-2}e^{-2a(a+b)/t}
\lesssim J^{(2)}_{\alpha,\beta}(t,a,b)\\\nonumber
&\lesssim \frac{t^{\alpha-1}}{a^{\alpha-\beta-1}(a+b)^{\alpha+\beta-1}}\(1+\frac{a(a+b)}{t}\)^{\alpha-2}e^{-a(a+b)/4t}.
\end{align}
If $\frac{a(a+b)}{t}\leq1$, \eqref{eq:J^1} and  \eqref{eq:J^2} give us
\begin{align}\label{eq:aux5}
J_{\alpha,\beta}(t,a,b)\approx \(\frac{a}{b}\)^{\beta-1}\frac{a}{a+b}\sqrt{\frac t{{ab}+t}}+\frac{t^{\alpha-1}}{a^{\alpha-\beta-1}(a+b)^{\alpha+\beta-1}}.
\end{align}
If $\frac{a(a+b)}{t}>1$, then
\begin{align*}
J^{(2)}_{\alpha,\beta}(t,a,b)&\lesssim \(\frac{a}{a+b}\)^\beta e^{-a(a+b)/8t} \(\(\frac{t}{a(a+b)}\)^{\alpha-1}\(1+\frac{a(a+b)}{t}\)^{\alpha-2}e^{-a(a+b)/8t}\)\\
&\lesssim \(\frac{a}{b}\)^{\beta-1}\frac{a}{a+b}\sqrt{\frac1{\frac{ab}{t}+1}}\lesssim J^{(1)}_{\alpha,\beta}(t,a,b),
\end{align*}
and consequently, using inequality $\frac {ab}t+1\leq 2a(a+b)$ and the assumption $\alpha\geq\frac32$,
\begin{align*}
J_{\alpha,\beta}(t,a,b)&\approx J^{(1)}_{\alpha,\beta}(t,a,b)\approx \(\frac{a}{b}\)^{\beta-1}\frac{a}{a+b}\sqrt{\frac1{\frac{ab}{t}+1}}\\
&\approx \(\frac{a}{b}\)^{\beta-1}\frac{a}{a+b}\sqrt{\frac1{\frac{ab}{t}+1}}\(1+\(\frac{b}{a+b}\)^{\beta-1}\(\frac t{a(a+b)}\)^{\alpha-1}\sqrt{\frac{ab}{t}+1}\)\\
&\approx \(\frac{a}{b}\)^{\beta-1}\frac{a}{a+b}\sqrt{\frac t{{ab}+t}}+\frac{t^{\alpha-1}}{a^{\alpha-\beta-1}(a+b)^{\alpha+\beta-1}},
\end{align*}
which is the same bound as in \eqref{eq:aux5}. Applying this to \eqref{eq:I=J+J} we obtain the desired estimate of $I_{\alpha,\beta}(a,v)$.

\end{proof}
The last result of this section is used in the paper for approximation of factors that come from the from of the heat kernel of a half-space \eqref{eq:kHform}.
\begin{proposition}
Let $c_1>0$ be a fixed constant. There exist another constant $c_0>0$ such that for every  $u,v>0$ satisfying $\frac{u}{v}>c_1$ we have
\begin{align}\label{eq:1-e}
\left|\frac{1-e^{-u}}{1-e^{-v}}-1\right|\leq c_0\frac{|u-v|}{v}.
\end{align}
\end{proposition}
\begin{proof}
First, let us recall two simple bound:
\begin{align*}
\begin{array}{rll}
|1-e^w|&\leq \ \ e(1\wedge |w|)(1+e^w), &w\in \R,\\
1-e^w&\geq \ \ (1-e^{-1})(1\wedge w),&w\geq0. 
\end{array}
\end{align*}
Using them, we get
\begin{align*}
\left|\frac{1-e^{-u}}{1-e^{-v}}-1\right|&=\left|e^{-v}\frac{1-e^{v-u}}{1-e^{-v}}\right|
\leq e\left|e^{-v}\frac{|v-u|(1+e^{v-u})}{1\wedge v}\right|
= e\left|(e^{-v}+e^{-u})\frac{|v-u|}{1\wedge v}\right|\\
&\leq \frac{|v-u|}{v}e\sup_{w\in(0,\infty)}\left\{\(e^{-v}+e^{-c_1v}\)\frac{v}{1\wedge v}\right\},
\end{align*}
as required.
\end{proof}

\end{document}